\documentclass[amstex,12pt]{amsart}
\usepackage{mathrsfs}
\usepackage{amssymb}
\usepackage{amsfonts}
\usepackage{amsfonts,amsmath,amsthm, amssymb}
\usepackage{array}
\usepackage{latexsym, euscript, epic, eepic}
\usepackage{tikz}
\usetikzlibrary{matrix}

\newtheorem{theorem}{Theorem}[section]
\newtheorem{lemma}[theorem]{Lemma}
\newtheorem{proposition}[theorem]{Proposition}
\newtheorem{corollary}[theorem]{Corollary}
\newtheorem{remark}[theorem]{Remark}
\newtheorem*{definition}{Definition}
\newtheorem*{claim}{Claim}

\begin{document}
 
\title[BMO Teich\-m\"ul\-ler spaces]{BMO Teich\-m\"ul\-ler spaces and their quotients \\with complex and metric structures} 

\author[H. Wei]{Huaying Wei} 
\address{Department of Mathematics and Statistics, Jiangsu Normal University \endgraf Xuzhou 221116, PR China} 
\email{hywei@jsnu.edu.cn} 

\author[K. Matsuzaki]{Katsuhiko Matsuzaki}
\address{Department of Mathematics, School of Education, Waseda University \endgraf
Shinjuku, Tokyo 169-8050, Japan}
\email{matsuzak@waseda.jp}

\thanks{Research supported by the National Natural Science Foundation of China (Grant No. 11501259)
and Japan Society for the Promotion of Science (KAKENHI 18H01125).}

\begin{abstract}
The paper presents some recent results on the BMO Teich\-m\"ul\-ler space, its subspaces and quotient spaces. We first consider the chord-arc curve subspace and prove that every element of the BMO Teich\-m\"ul\-ler space is represented by its finite composition. Moreover, we show that these BMO Teich\-m\"ul\-ler spaces have affine foliated structures induced by the VMO Teich\-m\"ul\-ler space. By which, their quotient spaces have natural complex structures modeled on the quotient Banach space. Then, a complete translation-invariant metric is introduced on the BMO Teich\-m\"ul\-ler space and is shown to be a continuous Finsler metric in a special case.
\end{abstract}

\maketitle

\section{Introduction}

The Teich\-m\"ul\-ler space is originally a universal classification space of the complex structures on a surface of given quasiconformal type, but according to complex analytic objects we focus on, we can also consider various kinds of Teich\-m\"ul\-ler spaces. The universal Teich\-m\"ul\-ler space plays a role of their ambient space, and its intrinsic natures (complex structures and invariant metrics) dominate any included Teich\-m\"ul\-ler spaces. For instance, the Teich\-m\"ul\-ler space of a Riemann surface can be represented in the universal Teich\-m\"ul\-ler space as the fixed point locus of the Fuchsian group. In a different direction to this, Teich\-m\"ul\-ler spaces in our study are obtained by adding a certain regularity to ingredients of the space. Recently, this type of Teich\-m\"ul\-ler spaces become more popular as a branch of infinite dimensional Teich\-m\"ul\-ler theory. 

The Bers model of the universal Teich\-m\"ul\-ler space $T$ is defined by the Schwarzian derivative $\mathcal{S}(f|_{\mathbb{D}^{*}})$ of the conformal homeomorphism $f$ of the exterior of the unit disk $\mathbb{D}^{*}$ that is quasiconformal on the unit disk $\mathbb{D}$. In this way, $T$ is embedded in a certain Banach space as a bounded domain. The image $\Gamma$ of the unit circle $\mathbb{S}$ under $f$ is called a quasicircle. The universal Teich\-m\"ul\-ler space $T$ can be also characterized as the set of all quasicircles up to a M\"obius transformations of the Riemann sphere $\widehat{\mathbb{C}}$. Let $\Omega$ denote the inner domain of $\Gamma$, and let moreover $g$ be a Riemann map of $\mathbb{D}$ onto $\Omega$. We define the conformal welding homeomorphism $h$ with respect to $\Gamma$ by $h = (g|_{\mathbb{S}})^{-1}\circ (f|_{\mathbb{S}})$, which is quasisymmetric. The universal Teich\-m\"ul\-ler space $T$ is identified with the group $\rm QS$ of quasisymmetric self-homeomorphisms $h$ of $\mathbb{S}$ modulo the group $\mbox{\rm M\"ob}(\mathbb{S})$ of M\"obius transformations of $\mathbb{S}$, i.e., $T = \mbox{\rm M\"ob}(\mathbb{S})\backslash {\rm QS}$. In general, the quasisymmetric homeomorphism $h$ does not satisfy any regularity conditions such as  absolute continuity. As well, the quasicircle $\Gamma$ might not even be rectifiable. In fact, its Hausdorff dimension, though less than $2$, can be arbitrarily close to $2$ (see \cite{BA, Bi, Se88}). 

The BMO theory has been studied often in the framework of Teich\-m\"ul\-ler theory. The corresponding subspaces of $T$ have generally satisfactory characteristics in terms of the quasicircle $\Gamma$ and the quasisymmetric homeomorphism $h$ (see \cite{AZ, Se, SW}). In this paper, we shall continue to study the BMO theory of the universal Teich\-m\"ul\-ler space, because of its great importance in the application to the harmonic analysis (see \cite{Da, FKP, Jo, Se}) and also of its own interest. We will especially focus on BMO Teich\-m\"ul\-ler spaces, the subspaces of the universal Teich\-m\"ul\-ler space $T$ closely related to BMO functions, Carleson measures and $A_{\infty}$ weights. In Section 2, we survey the standard theory of the universal Teich\-m\"ul\-ler space and BMO Teich\-m\"ul\-ler spaces. 

Basic problems are considered by going back and forth between the quasicircle $\Gamma$ and the conformal welding homeomorphism $h$ corresponding to $\Gamma$. It is known that the set $\rm SQS$ of all strongly quasisymmetric homeomorphisms of $\mathbb{S}$, which correspond to Bishop-Jones  quasicircles, forms a partial topological group under the BMO topology; the neighborhood base is given at the identity by using the BMO norm and is distributed at every point $h \in \rm SQS$ by the right translation. It is proved in \cite{We} that its characteristic topological subgroup $\rm SS$ consists of strongly symmetric homeomorphisms, which correspond precisely to asymptotically smooth curves in the sense of Pommerenke. We consider intermediately the set $\rm CQS$ of conformal welding homeomorphisms with respect to chord-arc curves $\Gamma$. In Section 3, we prove that every element of $\rm SQS$ can be represented as a finite composition of elements in $\rm CQS$ (Theorem \ref{represented}). As a consequence, we see that $\rm CQS$ does not carry a group structure under the composition (Corollary \ref{not}).

The Bers embedding of the universal Teich\-m\"ul\-ler space $T$ is a map into the Banach space of bounded holomorphic quadratic differentials. Affine foliated structures of $T$ and the quotient Bers embeddings are induced by its subspaces. This was first investigated by Gardiner and Sullivan \cite{GS} for the little subspace $T_0$, which consists of the asymptotically conformal elements of $T$. Later, it was proved that the Bers embedding is compatible with the coset decomposition $T_0\backslash T$ and the quotient Banach space. By which the complex structure modeled on the quotient Banach space is provided for $T_0\backslash T$ through the quotient Bers embedding. In Sections 4 and 5, we show two examples of affine foliated structures of the BMO Teich\-m\"ul\-ler spaces $T_c=\mbox{\rm M\"ob}(\mathbb{S})\backslash \rm CQS$ and $T_b=\mbox{\rm M\"ob}(\mathbb{S})\backslash \rm SQS$ (Corollary \ref{affine}, Theorem \ref{foliated}) and the injectivity of their quotient Bers embeddings by $T_v=\mbox{\rm M\"ob}(\mathbb{S})\backslash \rm SS$ (Corollary \ref{quotient}). 

In Section 6, a new invariant metric $m_C$ under the right translation is introduced on the BMO Teich\-m\"ul\-ler space by using the Carleson norm. We call this the Carleson metric. This is shown to be a continuous Finsler metric in the special case for $T_v$ (Theorem \ref{cont}). Moreover, the Carleson metric $m_C$ induces a quotient metric on the quotient BMO Teich\-m\"ul\-ler space. Then, a list of intended results is presented in this section, following the work on the asymptotic Teich\-m\"ul\-ler space $T_0\backslash T$ by Earle, Gardiner and Lakic \cite{EGL}. In the following Section 7, we show that the Carleson distance induced by $m_C$ is complete in the BMO Teich\-m\"ul\-ler spaces (Theorem \ref{complete}). We also compare the Carleson distance with the Teich\-m\"ul\-ler distance and the Kobayashi distance. 

One of our motivations to study those structures of BMO Teich\-m\"ul\-ler spaces is to consider an open problem of the connectivity of the  chord-arc curve subspace (see \cite{AZ}). The topology on this space is induced by the BMO norm of the conformal welding homeomorphisms.
The distribution of the chord-arc curve subspace $T_c$ in the BMO Teich\-m\"ul\-ler space $T_b$ (Theorem \ref{Tc}) can translate the problem of the connectivity to the quotient $T_v \backslash T_c$. By introducing the (quotient) Carleson metric in this space,
we can investigate a certain convexity of the chord-arc curve subspace to consider the problem.

\section{Preliminaries}

In this section, we review basic facts on the universal Teich\-m\"ul\-ler space (see \cite{Ah66, Le, Na88}) and the BMO theory of the universal Teich\-m\"ul\-ler space (see \cite{AZ, Se, SW}).

The universal Teich\-m\"ul\-ler space $T$ is defined as the group $\rm QS$ of all quasi\-symmetric homeo\-morphisms of the
unit circle ${\mathbb S}=\{z \mid |z|=1\}$ modulo the left action of
the group $\mbox{\rm M\"ob}(\mathbb{S})$ of all M\"obius transformations of $\mathbb{S}$, i.e., $T=\mbox{\rm M\"ob}(\mathbb{S}) \backslash \rm QS$.
A topology of $T$ can be defined by quasisymmetry constants of quasi\-symmetric homeo\-morphisms.
The universal Teich\-m\"ul\-ler space $T$ can be also defined by using quasiconformal homeo\-morphisms of the unit disk 
$\mathbb{D}=\{z \mid |z|<1\}$ with complex dilatations $\mu$
in the space of Beltrami coefficients 
$$
M(\mathbb D)=\{\mu \in L^\infty(\mathbb D) \mid \Vert \mu \Vert_\infty<1\}.
$$
Then, $T$ is the quotient space of $M(\mathbb D)$ under the Teich\-m\"ul\-ler equivalence.
The topology of $T$ coincides with the quotient topology induced by the projection $\pi:M(\mathbb D) \to T$.

The universal Teich\-m\"ul\-ler space $T$ is identified with a domain in the Banach space 
$$
B(\mathbb{D}^{*})=\{\varphi(z)dz^2 \mid \Vert \varphi \Vert_B=\sup_{z \in \mathbb D^*} (|z|^2-1)^2|\varphi(z)|<\infty\}
$$ 
of bounded holomorphic quadratic differentials on 
$\mathbb{D}^{*}=\widehat{\mathbb C}-\overline{\mathbb{D}}$ under the Bers embedding $\beta: T \to B(\mathbb{D}^{*})$.
This map is given by the factorization of a map $\Phi:M(\mathbb D) \to B(\mathbb{D}^{*})$ by the projection $\pi$, i.e.,
$\beta \circ \pi=\Phi$. Here, for every $\mu \in  M(\mathbb D)$, $\Phi(\mu)$ is defined by the Schwarzian derivative 
${\mathcal S}(f_\mu|_{\mathbb D^*})$
of the conformal homeomorphism $f_\mu$ of $\mathbb D^*$ that is quasiconformal on $\mathbb D$ with the complex dilatation $\mu$.
The Bers embedding $\beta: T \to B(\mathbb{D}^{*})$ is a homeomorphism onto the image $\beta(T)=\Phi(M(\mathbb D))$, 
and it defines a complex structure of $T$
as a domain in the Banach space $B(\mathbb{D}^{*})$.
It is proved that $\Phi$ (and so is $\pi$) is a holomorphic split submersion from $M(\mathbb{D})$ onto its image.

The space $M(\mathbb D)$ of Beltrami coefficients and the universal Teich\-m\"ul\-ler space $T$ are equipped with
a group structure. This can be easily seen by normalizing the elements of quasiconformal self-homeomorphisms of $\mathbb D$ and
quasisymmetric self-homeomorphisms of $\mathbb S$. The normalization is defined by fixing three distinct points (e.g., $1,i,-1$) of $\mathbb S$.
Then, $M(\mathbb D)$ is identified with the group of all normalized quasiconformal self-homeomorphisms of $\mathbb D$, and
$T$ is identified with the group of all normalized quasisymmetric self-homeo\-morphisms of $\mathbb S$.
The operation on the groups $M(\mathbb D)$ and $T$ is denoted by $\ast$. 
For every $\mu \in M(\mathbb D)$, the normalized quasiconformal self-homeomorphism of $\mathbb D$ 
with the complex dilatation $\mu$
(and its quasisymmetric extension to $\mathbb S$) is denoted by $f^{\mu}$. 
Then, the operation $\ast$ is defined by the relation $f^{\mu_1} \circ f^{\mu_2}=f^{\mu_1 \ast \mu_2}$.

For every $\nu \in M(\mathbb D)$, the right translation 
$r_\nu:M(\mathbb D) \to M(\mathbb D)$ is defined by $r_\nu(\mu)=\mu \ast \nu^{-1}$.
For every $\tau \in T$, the right translation 
$R_\tau:T \to T$ is defined by $R_\tau(\sigma)=\sigma \ast \tau^{-1}$.
Both $r_\nu$ and $R_\tau$ are biholomorphic automorphisms of $M(\mathbb D)$ and $T$, respectively.
Moreover, for $\tau=\pi(\nu)$, we have $R_\tau \circ \pi=\pi \circ r_\nu$.

A quasisymmetric homeomorphism $h \in \rm QS$ is called strongly quasisymmetric if
for any $\varepsilon>0$ there is some $\delta>0$ such that for any
arc $I \subset \mathbb{S}$ 
and any Borel set $E \subset I$,
$|E|\leqslant \delta |I|$ implies that $|h(E)| \leqslant \varepsilon |h(I)|$.
It should be noted that each $h$ is absolutely continuous and 
$\log h'$ is in ${\rm BMO}(\mathbb S)$. Here, a locally integrable function $\phi$ on $\mathbb S$ belongs to ${\rm BMO}(\mathbb S)$ if
$$
\Vert \phi \Vert_{\rm BMO}=\sup_{I \subset \mathbb S} \frac{1}{|I|} \int_I |\phi-\phi_I|\,\frac{d\theta}{2\pi}<\infty,
$$
where the supremum is taken over all arcs $I$ on $\mathbb S$, $|I| = \int_I d\theta/2\pi$ is the length of $I$, 
and $\phi_I$ denotes the average of $\phi$ over $I$.
We denote by $\rm SQS$ the group of all strongly quasisymmetric homeomorphisms. 
We assign the following BMO distance to $\rm SQS$: 
$d_{\rm BMO}(h_1,h_2)=\Vert \log h_1'-\log h_2' \Vert_{\rm BMO}$. 
The BMO Teich\-m\"ul\-ler space is defined by $T_b=\mbox{\rm M\"ob}(\mathbb{S}) \backslash \rm SQS$, which is equipped with a topology induced by
the BMO distance.

As in the case of the universal Teich\-m\"ul\-ler space, the BMO Teich\-m\"ul\-ler space $T_b$ has
the corresponding space for Beltrami coefficients.
For a simply connected domain $\Omega$ in the Riemann sphere $\widehat{\mathbb C}$ with $\infty \notin \partial \Omega$,
a measure $\lambda \doteq \lambda(z)dxdy$ on $\Omega$ is called a Carleson measure if 
$$
\Vert \lambda \Vert_c =\sup\Big\{ \frac{1}{r} \int_{|z-\zeta| < r,\ z \in \Omega} \lambda(z)dxdy \mid \zeta \in \partial \Omega,\ 0 < r < \rm diameter(\partial\Omega)\Big\}
$$
is finite. We denote the set of all Carleson measures on $\Omega$ by $ \rm CM(\Omega)$.
For any $\mu \in L^\infty(\mathbb D)$ and for the Poincar\'e density $\rho_{\mathbb D}(z)=(1-|z|^2)^{-1}$ (with curvature constant equal to $-4$) on $\mathbb{D}$, we set 
$$
\lambda_\mu(z)dxdy=|\mu(z)|^2 \rho_{\mathbb D}(z)dxdy.
$$ 
Then, a linear subspace ${\mathcal L}(\mathbb D) \subset L^\infty(\mathbb D)$ consisting of all $\mu$ with $\lambda_\mu \in \rm CM(\mathbb D)$
is a Banach space
with a norm 
$\Vert \mu \Vert_*=\Vert \mu \Vert_\infty+\Vert \lambda_\mu \Vert_c^{1/2}$.
We define ${\mathcal M}(\mathbb D)={\mathcal L}(\mathbb D) \cap M(\mathbb D)$. This corresponds to $T_b$
in such a way that $T_b$ is the image of ${\mathcal M}(\mathbb D)$ under the projection
$\pi:M(\mathbb D) \to T$
and the topology on $T_b$ induced from ${\mathcal M}(\mathbb D)$ by $\pi$
coincides with the topology on $T_b$ induced from the BMO distance.

There is also a subspace of bounded quadratic differentials corresponding to $T_b$.
For $\varphi \in B(\mathbb D^*)$, another norm is given by
$$
\Vert \varphi \Vert_{\mathcal B}=\Vert |\varphi(z)|^2\rho_{\mathbb D^*}^{-3}(z)dxdy \Vert_c^{1/2}
$$
as a Carleson measure on $\mathbb D^*$,
where $\rho_{\mathbb D^*}(z)=(|z|^2-1)^{-1}$ is
the Poincar\'e density on $\mathbb D^*$. We consider the Banach space
$\mathcal{B}(\mathbb D^*) \subset B(\mathbb D^*)$ consisting of all such elements $\varphi$ that 
$|\varphi(z)|^2\rho_{\mathbb D^*}^{-3}(z)dxdy \in \rm CM(\mathbb D^*)$
equipped with this norm. Then, it was proved in \cite[Theorem 5.1]{SW} that
$\Phi$ restricted to ${\mathcal M}(\mathbb D)$ is a holomorphic split submersion to $\mathcal{B}(\mathbb D^*)$ and
the Bers embedding $\beta$ of $T_b$ is a homeomorphism onto the domain 
$\beta(T_b)=\Phi({\mathcal M}(\mathbb D))$ in $\mathcal{B}(\mathbb D^*)$.
In particular, $T_b$ has a complex structure modeled on $\mathcal{B}(\mathbb D^*)$.

\section{Chord-arc curves do not have the group structure}
Let $\Gamma$ be a Jordan curve in the Riemann sphere $\widehat{\mathbb C}$, 
let $\Omega$ and $\Omega^{*}$ denote its inner and outer domains
in $\widehat{\mathbb C}$, respectively, and let $g$ and $f$ be conformal maps of $\mathbb{D}$ 
and $\mathbb{D}^{*}$ onto $\Omega$ and $\Omega^{*}$, respectively. We define the conformal welding homeomorphism $h$ with respect to $\Gamma$ by 
$h = (g|_{\mathbb{S}})^{-1}\circ (f|_{\mathbb{S}})$. 

A rectifiable Jordan curve $\Gamma$ in the complex plane $\mathbb C$ is called a chord-arc curve if 
$l_{\Gamma}(z_1, z_2) \leqslant K|z_1 - z_2|$ for any $z_1, z_2 \in \Gamma$, where $l_{\Gamma}(z_1, z_2)$ denotes 
the Euclidean length of the shorter arc of $\Gamma$ between $z_1$ and $z_2$. The smallest such $K$ is called the chord-arc constant for $\Gamma$. 
It is a well-known fact that a chord-arc curve is the image of $\mathbb{S}$ under 
a bi-Lipschitz homeomorphism $f$ of $\mathbb{C}$.
That is, there exists a homeomorphism $f: \mathbb{C} \to \mathbb{C}$ with a constant $C \geqslant 1$ such that 
$f(\mathbb{S}) = \Gamma$ and 
$C^{-1}|z - w| \leqslant |f(z) - f(w)| \leqslant C|z - w|$ for all $z, w \in \mathbb{C}$. 
When $\Gamma$ is a Jordan curve passing through $\infty$, we may replace the Euclidean distance in the definition above with
the spherical distance in order to define $\Gamma$ to be a chord-arc curve.
Bi-Lipschitz homeomorphisms preserve the Hausdorff dimension, and hence Hausdorff dimensions of chord-arc curves are one.

Although chord-arc curves are in a very special class of quasicircles, no characterization has been found in terms of their conformal welding homeomorphisms of $\mathbb{S}$. We denote the set of all these conformal welding homeomorphisms by $\rm CQS$. 
It is known that if $h \in \rm CQS$ then $h \in \rm SQS$ (see \cite{AZ}), that is, $h$ is strongly quasisymmetric, and in particular,
$\Vert \log h' \Vert_{\rm BMO}<\infty$. Conversely, there exists some constant $c>0$ such that if $\Vert \log h' \Vert_{\rm BMO}<c$
then $h \in \rm CQS$ and the corresponding $\Gamma$ is a chord-arc curve with the chord-arc constant $K$ sufficiently close to $\pi/2$.

In this section, we prove that every element of $\rm SQS$ can be represented as a finite composition of
elements in ${\rm CQS}$. As a consequence, we see that ${\rm CQS}$
does not carry a group structure under the composition. 
We state our results in the framework of Teich\-m\"ul\-ler theory. The chord-arc curve space is identified with
a subspace $T_c$ of
the BMO Teich\-m\"ul\-ler space $T_b$, which is given by
the set ${\rm CQS}$ modulo $\mbox{\rm M\"ob}(\mathbb{S})$, i.e., 
$T_c = \mbox{\rm M\"ob}(\mathbb{S}) \backslash {\rm CQS} \subset T_b$.
By regarding $T_c$ as a subset of the group $(T_b,\ast)$,
we can think of the inverse and the composition of elements of $T_c$.

For the proof of the main result in this section, we first claim that
$T_c$ (or ${\rm CQS}$) is preserved under the inverse.

\begin{proposition}\label{inverse}
The inverse element $\tau^{-1}$ belongs to $T_c$ for every $\tau \in T_c$.
\end{proposition}
\begin{proof}
If $h = g^{-1}\circ f$ is the conformal welding homeomorphism corresponding to a chord-arc curve $\Gamma$, then 
$h^{-1} = f^{-1}\circ g = (j\circ f\circ j)^{-1}\circ (j\circ g\circ j)$ is the conformal welding homeomorphism corresponding to $j(\Gamma)$, where $j(z) =z^*= \bar{z}^{-1}$ is the standard reflection of $\mathbb S$. 
By bi-Lipschitz continuity of $j$, $j(\Gamma)$ is a chord-arc curve. This proves that if $\tau=[h] \in T_c$ then $\tau^{-1} \in T_c$.
\end{proof}

\begin{remark}
{\rm
Noting that $\log(h^{-1})' = -\log h' \circ h^{-1}$, we conclude by Jones \cite{Jo} that 
$\|\log(h^{-1})'\|_{\rm BMO} \leqslant C\|\log h'\|_{\rm BMO}$ for some constant $C>0$
depending only on the strongly quasisymmetric constant for $h$.
Thus, the inverse mapping $T_c \ni \tau=[h] \mapsto \tau^{-1} \in T_c$ is continuous at the origin $o=[\rm id]$
under the BMO topology. However, this correspondence should not be continuous except at the origin.
}
\end{remark}

We now prove the claim mentioned above as follows.
\begin{theorem}\label{represented}
Each element of $T_b$ can be represented as a finite composition of elements in $T_c$.
\end{theorem}
\begin{proof}
Let $V$ denote a subset of $T_b$ consisting of all $\tau$ for which there exists an open neighborhood $W$ such that each $\tau' \in W$ 
can be represented as finite composition of elements in $T_c$.
Since $T_b$ is connected, in order to prove that $V$ coincides with $T_b$,
it suffices to show that $V$ is nonempty, open, and closed. 
$T_c$ is an open subset of $T_b$ containing the origin $o=[\rm id]$. This is essentially shown by Zinsmeister \cite{Zi85} (see also \cite{AZ}).  We see that
$o \in V$, and hence $V$ is nonempty. By the definition of $V$, $V$ is open.

Now we prove that $V$ is closed. Let $\{\tau_n\} \subset V$ be a sequence
such that $\tau_n \to \tau$ as $n \to \infty$. We will show that $\tau \in V$. 
Let $U$ be an open neighborhood of $o$ in $T_c$. 
Then, $R_{\tau}(\tau_n) \in U \subset T_c$ for all sufficiently large $n$. 
Hence, there exists an element $\sigma \in U \subset T_c$ 
such that $\tau_n \ast \tau^{-1} = \sigma$. By Proposition \ref{inverse}, we have $\tau \ast \tau_n^{-1}=\sigma^{-1} \in T_c$. 

We see that $R_{\tau}^{-1}(U)$ is a neighborhood of $\tau$. For each $\tau' \in R_{\tau}^{-1}(U)$, we have $R_{\tau}(\tau') \in U$. 
Namely, there exists an element $\sigma' \in U \subset T_c$ such that $\tau' \ast \tau^{-1} = \sigma'$. It follows that $\tau' = \sigma' \ast \tau = \sigma' \ast \sigma^{-1}\ast\tau_n$. Therefore, we have a neighborhood $R_{\tau}^{-1}(U) \doteq W$ of $\tau$ such that each $\tau' \in W$ can be represented as a finite composition of elements in $T_c$. This completes the proof.
\end{proof}

As $T_c \subsetneqq T_b$ by definition, we have the following immediate consequence from this theorem.

\begin{corollary}\label{not}
$T_c$ is not a subgroup of $(T,\ast)$.
\end{corollary}

\section{Foliated structure of the chord-arc curve  subspace}

We have mentioned that the chord-arc curve subspace $T_c$ is an open subset of $T_b$.
There is a long-standing open question about whether $T_c$ is connected or not.
For a recent account to a related result, see Astala and Gonz\'alez \cite{AG}. In this section, we prove
a result concerning the distribution of $T_c$ in $T_b$.

In the universal Teich\-m\"ul\-ler space $T$, there is a closed subspace $T_0$ defined by $T_0=\pi(M_0(\mathbb D))$,
where $M_0(\mathbb D)$ is the space of Beltrami coefficients vanishing on the boundary. The subspace $T_0$ can be also defined to be $\mbox{\rm M\"ob}(\mathbb{S})\backslash {\rm Sym}$ by the subgroup $\rm Sym \subset QS$ consisting of symmetric homeomorphisms of $\mathbb{S}$, which are the boundary extension of asymptotically conformal homeomorphisms of $\mathbb{D}$ whose complex dilatations belong to $M_0(\mathbb{D})$.  
We denote by $B_0(\mathbb{D^{*}})$ the Banach subspace of $B(\mathbb{D}^{*})$ consisting of all elements $\varphi$ such that 
$\rho_{\mathbb D^*}^{-2}(z)|\varphi(z)| \to 0$ as $|z| \to 1^+$. By the Bers embedding $\beta: T \to B(\mathbb{D}^{*})$, 
$T_0$ is mapped into $B_0(\mathbb{D}^{*})$ and identified with a domain $\beta(T_0)=\Phi(M_0(\mathbb D))$. 

Similarly, there is a closed subspace in $T_b$ that can be given by vanishing Carleson measures on $\mathbb D$.
Here, we say that a Carleson measure $\lambda(z)dxdy$ on a simply connected domain $\Omega$ is vanishing if 
$$
\lim_{r \to 0} \frac{1}{r} \int_{|z-\zeta| < r,\ z \in \Omega} \lambda(z)dxdy=0
$$
uniformly for $\zeta \in \partial \Omega$. The set of all such vanishing Carleson measures on $\Omega$ is denoted by $\rm CM_0(\Omega)$.
Let ${\mathcal M}_0(\mathbb D)$ be the subspace of ${\mathcal M}(\mathbb D)$
consisting of all Beltrami coefficients $\mu$ such that $\lambda_\mu(z)dxdy \in \rm CM_0(\mathbb D)$. Then,
$T_v=\pi({\mathcal M}_0(\mathbb D))$ is a closed subspace of $T_b$, which is called
the VMO Teich\-m\"ul\-ler space. 
We denote $\mathcal{B}_0(\mathbb{D}^{*})$ by the Banach subspace of $\mathcal{B}(\mathbb{D}^{*})$ consisting of all elements $\varphi$ 
such that $|\varphi(z)|^2\rho_{\mathbb D^*}^{-3}(z)dxdy \in \rm CM_0(\mathbb{D}^{*})$. 
Then $\mathcal{B}_0(\mathbb{D}^{*}) \subset B_0(\mathbb{D}^{*})$ by \cite[Lemma 4.1]{SW}.
It is proved in \cite[Theorems 4.1, 5.2]{SW} that $\Phi$ maps $\mathcal{M}_0(\mathbb{D})$ into $\mathcal{B}_0(\mathbb{D}^{*})$ and the Bers embedding $\beta$ of $T_v$ is a homeomorphism onto a domain $\beta(T_v)=\Phi({\mathcal M}_0(\mathbb D))$ in $\mathcal{B}_0(\mathbb{D}^{*})$. 

The VMO Teich\-m\"ul\-ler space $T_v$ can be also defined to be $T_v=\mbox{\rm M\"ob}(\mathbb S)\backslash \rm SS$ by the characteristic topological subgroup $\rm SS$ of the partial topological group $\rm SQS$
consisting of all strongly symmetric homeomorphisms. 
Here, we say that $h \in \rm SQS$ is strongly symmetric if $\log h' \in {\rm VMO}(\mathbb S)$, where
a function $\phi \in {\rm BMO}(\mathbb S)$ belongs to ${\rm VMO}(\mathbb S)$ if 
$$
\lim_{|I| \to 0} \frac{1}{|I|} \int_I |\phi - \phi_I|\,\frac{d\theta}{2\pi}=0
$$
uniformly. In fact, ${\rm VMO}(\mathbb{S})$ is the closed subspace of ${\rm BMO}(\mathbb S)$ which is precisely the closure of the space of all continuous functions on $\mathbb{S}$ under the BMO topology. The inclusion relation $\rm SS \subset \rm Sym$ is known.

We prove that $T_c$ is distributed in $T_b$ entirely in all directions of $T_v$ in the following sense.
\begin{theorem}\label{Tc}
For each $\tau \in T_c$, we have  $R_{\tau}^{-1}(T_v) \subset T_c$. Hence, $T_c=\bigsqcup_{[\tau] \in T_v\backslash T_c}R_{\tau}^{-1}(T_v)$.
\end{theorem}

\begin{proof}
For each $\sigma \in T_v$, we will show that $\widehat{\sigma} \doteq \sigma\ast\tau$ belongs to $T_c$. Let $g^{-1}\circ f$ and $g_1^{-1}\circ f_1$ be the conformal welding homeomorphisms such that $[g^{-1}\circ f]=\tau$ and $[g_1^{-1}\circ f_1]=\widehat{\sigma}$.
We set $\Omega = g(\mathbb{D})$, $\Omega^{*} = f(\mathbb{D}^{*})$, and $\Gamma=\partial \Omega=\partial \Omega^*$ which is a chord-arc curve.
Similarly, we set $\Omega_1 = g_1(\mathbb{D})$ and $\Omega_1^{*} = f_1(\mathbb{D}^{*})$ 
with a quasicircle $\Gamma_1=\partial \Omega_1=\partial \Omega_1^*$. 
Let $f^{\nu}$ and $f^{\mu}$ be normalized quasiconformal self-homeomorphisms of $\mathbb D$
corresponding to 
$\tau$ and $\sigma$, respectively. Noting that $\sigma \in T_v$, we can assume that the complex dilatation $\mu$ of $f^{\mu}$ induces a vanishing Carleson measure $\lambda_{\mu} \in \rm CM_0(\mathbb{D})$. 
Then, $g\circ f^{\nu}$ and $g_1\circ f^{\mu}\circ f^{\nu}$ are quasiconformal extensions of $f$ and $f_1$ to $\mathbb{D}$, respectively. 

We define 
$$
\widehat{f} = \begin{cases}
f_1 \circ f^{-1} &  {\rm on}\;\; \Omega^{*}\\
(g_1 \circ f^{\mu}\circ f^{\nu})\circ (g\circ f^{\nu})^{-1} = g_1\circ f^{\mu}\circ g^{-1} & {\rm on}\;\; \Omega.
\end{cases}
$$
Then, $\widehat{f}: \widehat{\mathbb{C}} \to \widehat{\mathbb{C}}$ is conformal on $\Omega^{*}$ and asymptotically conformal on $\Omega$ whose complex dilatation $\widehat{\mu}$ satisfies $|\widehat{\mu}|^2\rho_{\Omega} = \lambda_{\mu}\circ g^{-1}|(g^{-1})^{'}|$ for the Poincar\'e density
$\rho_\Omega$ on $\Omega$. As $\lambda_\mu \in \rm CM_0(\mathbb D)$, we have that $|\widehat{\mu}|^2\rho_{\Omega} \in \rm CM_0(\Omega)$ by \cite[Theorem 3.2]{WZ}.

We decompose $\widehat{f}$ into $\widehat{f_0}\circ \widehat{f_1}$ as follows. The quasiconformal homeomorphism $\widehat{f_1}: \widehat{\mathbb{C}} \to \widehat{\mathbb{C}}$ is chosen so that its complex dilatation $\widehat \mu_1$ coincides with $\widehat{\mu}$ on $\Omega - \Omega_0$ 
for some compact subset $\Omega_0$ of $\Omega$ and zero elsewhere. Then $\widehat{f_0}$ is defined to be $\widehat{f}\circ\widehat f_1^{-1}$. We have the following commutative diagram:\\
%%%%%%%%%%%%%
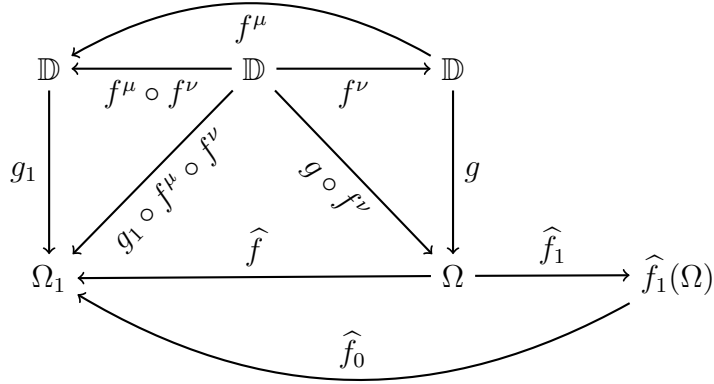
\begin{center}
\begin{tikzpicture}
\matrix(m)[matrix of math nodes, row sep=5em, column sep=5em]
{ |[name=ka]| \mathbb{D} & |[name=kb]| \mathbb{D} & |[name=kc]| \mathbb{D}\\
|[name=kd]| \Omega_1 & |[name=ke]|  & |[name=kf]| \Omega & |[name=kg]| \widehat{f_1}(\Omega) \\};

\draw[->, thick] 
(kb) edge node[below] {$f^{\mu}\circ f^{\nu}$} (ka)
(kb) edge node[below] {$f^{\nu}$} (kc)
(kc) edge node[auto] {$g$} (kf)
(kb) edge node[below, sloped, thick] {$g_1\circ f^{\mu}\circ f^{\nu}$} (kd)
(kb) edge node[below, sloped, thick] {$g\circ f^{\nu}$} (kf)
(ka) edge node[left] {$g_1$} (kd)
(kf) edge node[above] {$\widehat{f}$} (kd)
(kf) edge node[above] {$\widehat{f_1}$} (kg)
                          ;
\draw[->, thick][bend right] (kc) edge node[auto] {$f^{\mu}$} (ka);
\draw[->, thick][bend left] (kg) edge node[above] {$\widehat{f_0}$} (kd);                          
\end{tikzpicture}\\
\end{center}
%%%%%%%%%%%%%

Here, the compact subset $\Omega_0 \subset \Omega$ is chosen so that $|\widehat \mu_1|^2\rho_{\Omega} \in \rm CM_0(\Omega)$ has
a sufficiently small norm of the Carleson measure. It follows from \cite[Lemma 4.1]{WZ} that $|\mathcal{S}(\widehat{f_1})|^2\rho_{\Omega^{*}}^{-3} \in \rm CM_0(\Omega^{*})$
with a small norm. By \cite[Theorem 3.1]{WZ}, we have that
\begin{equation*}
|\mathcal{S}(\widehat{f_1}\circ f) - \mathcal{S}(f)|^2 \rho_{\mathbb D^{*}}^{-3} = (|\mathcal{S}(\widehat{f_1})|^2 \rho_{\Omega^{*}}^{-3})\circ f |f'| \in \rm CM_0(\mathbb{D}^{*}),    
\end{equation*}
and moreover we see that it can be of a small norm according to that of $|\mathcal{S}(\widehat{f_1})|^2\rho_{\Omega^{*}}^{-3}$. 
Combined with the facts that $\Gamma$ is a chord-arc curve and that the subspace $T_c$ is open, it implies that $\partial \widehat{f_1}(\Omega)$ is also a chord-arc curve. Since the complex dilatation $\widehat \mu_0$ of $\widehat{f_0}$ has the compact support $\widehat{f_1}(\Omega_0) \subset \widehat{f_1}(\Omega)$, we conclude that $\Gamma_1$ is the image of $\partial \widehat{f_1}(\Omega)$ under the conformal mapping $\widehat{f_0}$ defined on 
$\widehat{\mathbb{C}}- \widehat{f_1}(\Omega_0)$,
which is bi-Lipschitz in a neighborhood of $\partial \widehat{f_1}(\Omega)$. Thus, we see that $\Gamma_1$ is again a chord-arc curve, which implies that $\widehat{\sigma} \in T_c$.
\end{proof}

We consider the projection
$$
p:T_b=\mbox{\rm M\"ob}(\mathbb S) \backslash {\rm SQS} \to T_v \backslash T_b={\rm SS} \backslash {\rm SQS}.
$$
The quotient space $T_v \backslash T_b$ is endowed with the quotient topology.
We apply this quotient map to the subspace $T_c$. Then,
Theorem \ref{Tc} is equivalent to saying that $T_c=p^{-1}(p(T_c))$.
Concerning the topology of $T_c$ and $p(T_c)$, we immediately see the following.
Noting the fact that $T_v$ is contractible shown in \cite{TWS},
the connectedness problem on $T_c$ can be also passed to this quotient.
\begin{corollary}
The quotient space $p(T_c)$ is a proper open subset of $p(T_b)$. If $p(T_c)$ is connected, then so is $T_c$.
\end{corollary}

The quotient Bers embedding from $T_v\backslash T_c=p(T_c)$ into $\mathcal{B}_0(\mathbb{D}^{*})\backslash \mathcal{B}(\mathbb{D}^{*})$ is 
considered in \cite[Theorem 2.2]{WZ} to be 
well-defined and injective. We also generalize this theorem to the entire space $T_v \backslash T_b=p(T_b)$ in the next section.
Combining the claim for $T_c$ with Theorem \ref{Tc}, we have the following result naturally.

\begin{corollary}\label{affine}
$\beta\circ R_{\tau}^{-1}(T_v) = \beta (T_c) \cap \{\mathcal{B}_0(\mathbb{D}^{*}) + \beta (\tau)\}$ for  every $\tau \in T_c$.
\end{corollary}

By this result, we have the decomposition of the Bers embedding as 
\begin{equation*}
\beta(T_c) = \bigsqcup_{[\tau] \in T_v\backslash T_c}\beta\circ R_{\tau}^{-1}(T_v) = \bigsqcup_{[\psi] \in \mathcal{B}_0(\mathbb{D}^{*})\backslash \mathcal{B}(\mathbb{D}^{*})}\beta (T_c) \cap (\mathcal{B}_0(\mathbb{D}^{*}) + \psi).
\end{equation*}
Each component $\beta (T_c) \cap (\mathcal{B}_0(\mathbb{D}^{*}) + \psi)$ is biholomorphically equivalent to $T_v \cong \beta(T_v)$. We call this decomposition the affine foliated structure of $T_c$ induced by $T_v$.

\section{The quotient Bers embedding of the BMO Teich\-m\"ul\-ler space}

In this section, we prove the affine foliated structure of the BMO-Teich\-m\"ul\-ler space $T_b$ and the injectivity of the quotient Bers embedding induced by the VMO-Teich\-m\"ul\-ler space $T_v$. From this result, we provide the quotient BMO Teich\-m\"ul\-ler space $T_v \backslash T_b=p(T_b)$
with a complex structure modeled on the quotient Banach space $\mathcal{B}_0(\mathbb{D}^{*})\backslash \mathcal{B}(\mathbb{D}^{*})$.

\begin{theorem}\label{foliated}
$\beta\circ R_{\tau}^{-1}(T_v) = \beta (T_b) \cap \{\mathcal{B}_0(\mathbb{D}^{*}) + \beta (\tau)\}$ for every $\tau \in T_b$.
\end{theorem}
\begin{proof}
For every $\tau \in T_b= \mbox{\rm M\"ob}(\mathbb{S}) \backslash \rm SQS$, let $f^{\nu}: \mathbb{D}\to \mathbb{D}$ be a normalized
quasiconformal extension of $\tau$ with complex dilatation $\nu \in \mathcal{M}(\mathbb{D})$ (i.e. $\pi(\nu)=\tau$)
that is bi-Lipschitz under the Poincar\'e metric on $\mathbb D$ (for instance the Douady-Earle extension of $\tau$; see \cite{CZ}), and let $\psi = \Phi(\nu) \in \mathcal{B}(\mathbb{D}^{*})$.

For one inclusion $\subset$, we divide the arguments into two steps. We first deal with the special case that $\mu \in \mathcal{M}_0(\mathbb{D})$ has a compact support. Then, we extend this to the general case by means of an approximation process.

We take a Beltrami coefficient $\mu$ on $\mathbb D$ with compact support. 
Clearly, $\mu \in \mathcal{M}_0(\mathbb{D})$. We will show that
$\Phi(\mu * \nu) - \Phi(\nu) \in \mathcal{B}_0(\mathbb{D}^{*})$. Then, the inclusion $\subset$ follows from
$$
\Phi(\mu * \nu) - \Phi(\nu)=\beta \circ \pi(\mu * \nu)-\beta \circ \pi(\nu)=\beta \circ R^{-1}_\tau(\pi(\mu))-\beta (\tau).
$$
Let $f_{\nu}:  \widehat{\mathbb{C}} \to \widehat{\mathbb{C}}$ be a quasiconformal homeomorphism with complex dilatation $\nu \in \mathcal{M}(\mathbb{D})$ that is conformal on $\mathbb{D}^{*}$ with $\mathcal{S}_{f_{\nu}|_{\mathbb{D}^{*}}} = \psi$. 
Let $f_{\mu * \nu}: \widehat{\mathbb{C}} \to \widehat{\mathbb{C}}$ be a quasiconformal homeomorphism with complex dilatation $\mu * \nu$ on $\mathbb{D}$ that is conformal on $\mathbb{D}^{*}$. Set $\widehat{f} = f_{\mu * \nu} \circ f_{\nu}^{-1}$. Then, $\widehat{f}$ is a quasiconformal homeomorphism with complex dilatation $\widehat{\mu}$ on $\Omega = f_{\nu}(\mathbb{D})$ with a compact support contained in a Jordan domain $\Omega_0$ with $\overline{\Omega_0} \subset \Omega$,
and is conformal on $\Omega^{*} = f_{\nu}(\mathbb{D}^{*})$ with 
\begin{equation}\label{S}
 |\mathcal{S}({\widehat{f}})|^2 \rho_{\Omega^{*}}^{-3}= \Big(|\mathcal{S}(f_{\mu * \nu}) - \mathcal{S}(f_{\nu})|^2\rho_{\mathbb D^{*}}^{-3}\Big)\circ f_{\nu}^{-1}|(f_{\nu}^{-1})'|. \tag{$\ast$}
\end{equation}
We note that $\widehat{f}$ is conformal on $\Omega_{0}^{*} = \widehat{\mathbb{C}} - \overline{\Omega_0}$. 

It is known that $|\mathcal{S}(\widehat{f})(z)|\rho^{-2}_{\Omega_{0}^{*}}(z) \leqslant 12$ for $z \in \Omega_{0}^{*}$ (see \cite[p.67]{Le}).
Combined with the monotone property $\rho_{\Omega_0^*}(z) \leqslant \rho_{\Omega^*}(z)$ 
of Poincar\'e densities,
this inequality implies that there exists a constant $C$ such that 
\begin{equation*}
|\mathcal{S}(\widehat{f})(z)|^2 \rho_{\Omega^{*}}^{-3}(z) \leqslant 144 \rho^{4}_{\Omega_{0}^{*}}(z)\rho^{-3}_{\Omega^{*}}(z) \leqslant 144 \rho_{\Omega_{0}^{*}}(z) \leqslant C
\end{equation*}
for $z \in \Omega^{*}$. From which we deduce that $|\mathcal{S}(\widehat{f})|^2 \rho_{\Omega^{*}}^{-3} \in \rm  CM_0(\Omega^{*})$. By (\ref{S}) and well-definedness of the pull-back operator from $\rm CM_0(\Omega^{*})$ into $\rm CM_0(\mathbb{D}^{*})$ (see \cite[Theorem 3.1]{WZ}), we have that
$|\mathcal{S}(f_{\mu * \nu}) - \mathcal{S}(f_{\nu})|^2\rho_{\mathbb D^*}^{-3} \in \rm CM_0(\mathbb{D}^{*})$, which yields that $\Phi(\mu * \nu) - \Phi(\nu) \in \mathcal{B}_0(\mathbb{D}^{*})$.

For any $\sigma \in T_v = \mbox{\rm M\"ob}(\mathbb{S}) \backslash \rm SS$, the complex dilatation of the Douady-Earle extension of $\sigma$ is denoted by $\mu$. Then, $\mu \in \mathcal{M}_0(\mathbb{D})$ by \cite[Theorem 3.7]{TWS} (see also \cite{QW}). We take an increasing sequence of positive numbers $r_n < 1$ $(n = 1, 2, \ldots)$ tending to 1. Let $\Delta_n = \mathbb{D}(0, r_n)$, the disk of radius $r_n$ centered at the origin, and let $A_n = \mathbb{D} - \overline \Delta_n$. We define 
$$
\mu_n = \begin{cases}
\mu &  {\rm on}\;\; \overline{\Delta_n}\\
0 & {\rm on}\;\; A_n.
\end{cases}
$$
Then, $\{\mu_n\}$ is a sequence of complex dilatations with compact support such that 
\begin{equation*}
\begin{split}
\|\mu - \mu_n\|_{\ast} & =\|\mu - \mu_n\|_{\infty} + \|\lambda_{\mu - \mu_n}\|_{c}\\
& = \|\mu|_{A_n}\|_{\infty} + \|\lambda_{\mu|_{A_n}}\|_{c} \to 0\\
\end{split}
\end{equation*}
as $n \to \infty$. Indeed, it was proved in \cite{EMS} that the complex dilatation of the Douady-Earle extension of a symmetric homeomorphism is in $M_0(\mathbb{D})$. Combined with the inclusion relation $\rm SS \subset \rm Sym$, we see that $\mu$ belongs to $M_0(\mathbb{D})$, which yields that the first term tends to $0$. 
By the definition of $\mathcal{M}_0(\mathbb{D})$, we have that the second term tends to $0$. 

Since $f^{\nu}$ is bi-Lipschitz under the Poincar\'e metric, $\nu$ induces a biholomorphic automorphism $r_{\nu}^{-1}: \mathcal{M}(\mathbb{D}) \to \mathcal{M}(\mathbb{D})$ (see \cite[Remark 5.1]{SW}). 
Then, we have 
$$
\|r_{\nu}^{-1}(\mu) - r_{\nu}^{-1}(\mu_n)\|_\ast=\|\mu \ast \nu - \mu_n \ast \nu\|_\ast \to 0
$$
as $n \to \infty$. 
The continuity of $\Phi$ yields that 
$$
\|(\Phi(\mu * \nu) - \Phi(\nu)) - (\Phi(\mu_n * \nu) - \Phi(\nu))\|_{\mathcal{B}} = \|\Phi(\mu * \nu) - \Phi(\mu_n * \nu)\|_{\mathcal{B}} \to 0
$$
as $n \to \infty$. We have proved that $\Phi(\mu_n * \nu) - \Phi(\nu) \in \mathcal{B}_0(\mathbb{D}^{*})$ in the first step. Then, it follows from the fact that $\mathcal{B}_0(\mathbb{D}^{*})$ is closed in $\mathcal{B}(\mathbb{D}^{*})$ that 
$\Phi(\mu * \nu) - \Phi(\nu) \in \mathcal{B}_0(\mathbb{D}^{*})$. This proves the inclusion $\subset.$

For the other inclusion $\supset$, it can be proved by using the following claim, which is shown in \cite[Proposition 2.2]{Ma07}. 
\begin{claim}\label{M}
{\rm
Let $f_{\nu}: \widehat{\mathbb{C}} \to \widehat{\mathbb{C}}$ be a quasiconformal homeomorphism with complex dilatation 
$\nu \in M(\mathbb{D})$
that is bi-Lipschitz between $\mathbb D$ and $\Omega=f_{\nu}(\mathbb D)$ under their Poincar\'e metrics, 
and is conformal on $\mathbb{D}^{*}$ with $\mathcal{S}(f_{\nu}|_{\mathbb{D}^{*}}) = \psi$.  
Then, for every $\varphi \in  B_0(\mathbb{D}^{*})$, there exists a quasiconformal homeomorphism 
$\widehat{f}: \widehat{\mathbb{C}} \to \widehat{\mathbb{C}}$ with complex dilatation $\widehat{\mu}$ on $\Omega$ vanishing at the boundary
that is conformal on $\Omega^{*} = f_{\nu}(\mathbb{D}^{*})$ with $\mathcal{S}(\widehat{f}\circ f_{\nu}|_{\mathbb{D}^{*}}) = \varphi + \psi$ such that the following statements are valid: $\widehat{f}$ is decomposed into two quasiconformal homeomorphisms $\widehat{f_0}$ and $\widehat{f_1}$ of $\widehat{\mathbb{C}}$ with $\widehat{f} = \widehat{f_0}\circ\widehat{f_1}$, where $\widehat{f_1}$ is conformal on $\Omega^{*}$ with $\mathcal{S}(\widehat{f_1}\circ f_{\nu}|_{\mathbb{D}^{*}}) = \varphi_1 + \psi$, satisfying the following properties:
\begin{itemize}
\item[(1)] the complex dilatation $\widehat \mu_1$ of $\widehat f_1$ on $\Omega$ satisfies 
\begin{equation*}
|\widehat \mu_1 \circ f_{\nu}(z)| \leqslant \frac{1}{\varepsilon} \rho_{\mathbb{D}^{*}}^{-2}(z^{*})|\varphi_1(z^{*})| \qquad (z^{*} = \bar{z}^{-1})  
\end{equation*}
for some $\varepsilon > 0$ and for every $z \in \mathbb{D}$;
\item[(2)] the support of the complex dilatation $\mu_0$ of the normalized quasiconformal homeo\-morphism $f_0: \mathbb{D} \to \mathbb{D}$, which is conformally conjugate to $\widehat{f_0}: \widehat{f_1}(\Omega)\to \widehat{f}(\Omega)$, is contained in a compact subset of $\mathbb{D}$;
\item[(3)] for the complex dilatation $\mu_1$ of the normalized quasiconformal homeomorphism $f_1: \mathbb{D}\to \mathbb{D}$, which is conformally conjugate to $\widehat{f_1}: \Omega \to \widehat{f_1}(\Omega)$, we have 
\begin{equation*}
\varphi - \varphi_1 = \Phi(\mu_0 * \mu_1 * \nu) - \Phi(\mu_1 * \nu).
\end{equation*}
\end{itemize}
}
\end{claim}
Combining all those maps in the claim above, we have the following commutative diagram, where $g_\nu$, $g_1$, and $g$ are
the conjugating conformal maps:
%%%%%%%%%%%%%
\begin{center}
\begin{tikzpicture}
\matrix(m)[matrix of math nodes, row sep=5em, column sep=5em]
{ |[name=ka]| \mathbb{D} & |[name=kb]| \mathbb{D} & |[name=kc]| \mathbb{D} & |[name=kd]| \mathbb{D}\\
|[name=ke]|  & |[name=kf]| \Omega  & |[name=kg]| \widehat{f_1}(\Omega) & |[name=kh]| \widehat{f}(\Omega) \\};

\draw[->, thick] 
(ka) edge node[below] {$f^{\nu}$} (kb)
(kb) edge node[below] {$f_1$} (kc)
(kc) edge node[below] {$f_0$} (kd)
(ka) edge node[below, sloped, thick] {$f_{\nu}$} (kf)
(kf) edge node[above] {$\widehat{f_1}$} (kg)
(kg) edge node[above] {$\widehat{f_0}$} (kh)
(kb) edge node[right] {$g_{\nu}$} (kf)
(kc) edge node[right] {$g_1$} (kg)
(kd) edge node[right] {$g$} (kh)
                          ;
\draw[->, thick][bend left] 
(kb) edge node[below] {$f = f_1\circ f_0$} (kd);
\draw[->, thick][bend right] 
(kf) edge node[above] {$\widehat{f} = \widehat{f_1}\circ \widehat{f_0}$} (kh);                          
\end{tikzpicture}\\
\end{center}
%%%%%%%%%%%%%
We take $\varphi\in \mathcal{B}_0(\mathbb{D}^{*})$ such that $\varphi + \psi \in \beta(T_b)$. Since $\mathcal{B}_0(\mathbb{D}^{*}) \subset B_0(\mathbb{D}^{*})$, there is a quasiconformal homeomorphism $\widehat{f}: \widehat{\mathbb{C}} \to \widehat{\mathbb{C}}$ conformal on $\Omega^{*}$ and asymptotically conformal on $\Omega$ such that $\mathcal{S}(\widehat{f}\circ f_{\nu}|_{\mathbb{D}^{*}}) = \varphi + \psi$. According to the claim above, we consider the decomposition $\widehat{f} = \widehat{f_0}\circ\widehat{f_1}$ together with other maps that appear in it,
and apply the properties shown there. 

Since $\varphi \in \mathcal{B}_0(\mathbb{D}^{*})$, if $\varphi - \varphi_1 \in \mathcal{B}_0(\mathbb{D}^{*})$, then $\varphi_1 \in \mathcal{B}_0(\mathbb{D}^{*})$. By property $(2)$, $\mu_0$ in particular belongs to $\mathcal{M}_0(\mathbb{D})$, and property $(3)$ asserts that 
$\varphi - \varphi_1 =  \Phi(\mu_0 * \mu_1 * \nu) - \Phi(\mu_1 * \nu)$. 
By the previous arguments showing the inclusion $\subset$, we see that $\varphi - \varphi_1 \in \mathcal{B}_0(\mathbb{D}^{*})$. 
Hence, $\varphi_1 \in \mathcal{B}_0(\mathbb{D}^{*})$.

By property $(1)$, $\varphi_1 \in \mathcal{B}_0(\mathbb{D}^{*})$ implies that $\widehat \mu_1 \circ f_{\nu} \in \mathcal{M}_0(\mathbb{D})$
(see Section 6 for details). 
By $|\widehat \mu_1 \circ f_{\nu}| = |\mu_1 \circ f^{\nu}|$, we have $\mu_1 \circ f^{\nu} \in \mathcal{M}_0(\mathbb{D})$. It follows from the bi-Lipschitz continuity of $f^{\nu}$ and \cite[Proposition 3.5]{TWS} that $\mu_1 \in \mathcal{M}_0(\mathbb{D})$. By property $(2)$, the support of the complex dilatation $\mu_0$ of $f_0$ is contained in a compact subset of $\mathbb{D}$. Hence, we see that the complex dilatation $\mu_{f} = \mu_0 * \mu_1$ of $f = f_0 \circ f_1$ belongs to $\mathcal{M}_0(\mathbb{D})$. Since the complex dilatation 
of the quasiconformal homeomorphism $\widehat{f}\circ f_{\nu}$ 
on $\mathbb{D}$ is $r_{\nu}^{-1}(\mu_{f})$, we have that 
$$
\varphi + \psi = \Phi (\mu_0 * \mu_1 * \nu) = \Phi\circ r_{\nu}^{-1}(\mu_{f}) \in \Phi\circ r_{\nu}^{-1}(\mathcal{M}_0(\mathbb{D}))=\beta\circ R_{\tau}^{-1}(T_v) ,
$$
which proves the inclusion $\supset$.
\end{proof}

By this theorem, we have the decomposition of the Bers embedding as 
\begin{equation*}
\beta(T_b) = \bigsqcup_{[\tau] \in T_v\backslash T_b}\beta\circ R_{\tau}^{-1}(T_v) = \bigsqcup_{[\psi] \in \mathcal{B}_0(\mathbb{D}^{*})\backslash \mathcal{B}(\mathbb{D}^{*})}\beta (T_b) \cap (\mathcal{B}_0(\mathbb{D}^{*}) + \psi).
\end{equation*}
Each component $\beta (T_b) \cap (\mathcal{B}_0(\mathbb{D}^{*}) + \psi)$ is biholomorphically equivalent to $T_v \cong \beta(T_v)$. This is the affine foliated structure of $T_b$ induced by $T_v$.

From Theorem \ref{foliated}, we also see that the quotient space $T_v\backslash T_b$ can be identified with a domain in
the quotient Banach space $\mathcal{B}_0(\mathbb{D}^{*})\backslash \mathcal{B}(\mathbb{D}^{*})$. 

\begin{corollary}\label{quotient}
The quotient Bers embedding 
\begin{equation*}
 \widehat{\beta} : T_v\backslash T_b \to \mathcal{B}_0(\mathbb{D}^{*})\backslash \mathcal{B}(\mathbb{D}^{*})  
 \end{equation*}
is well-defined and injective. Moreover, $\widehat{\beta}$ is a homeomorphism of $T_v\backslash T_b$ onto its image. Consequently,  
$T_v\backslash T_b$ possesses a complex structure such that $\widehat{\beta}$ is a biholomorphic automorphism  from $T_v\backslash T_b$ onto its image. 
\end{corollary}
\begin{proof}
The well-definedness and injectivity of the map 
\begin{equation*}
 \widehat{\beta} : T_v\backslash T_b \to \mathcal{B}_0(\mathbb{D}^{*})\backslash \mathcal{B}(\mathbb{D}^{*})  
 \end{equation*}
are direct consequences from Theorem \ref{foliated}. 

We will show that the quotient Bers embedding  $\widehat{\beta}$ is also a homeomorphism from $T_v\backslash T_b$ onto its image. 
For the quotient maps $p: T_b \to T_v\backslash T_b$ and $P: \mathcal{B}(\mathbb{D}^{*}) \to \mathcal{B}_0(\mathbb{D}^{*})\backslash \mathcal{B}(\mathbb{D}^{*})$, the following commutative diagram holds:

%%%%%%%%%%%%%
\begin{center}
\begin{tikzpicture}
\matrix(m)[matrix of math nodes, row sep=5em, column sep=5em]
{ |[name=ka]| T_b & |[name=kb]| \mathcal{B}(\mathbb{D}^{*})\\
|[name=kc]| T_v\backslash T_b & |[name=kd]| \mathcal{B}_0(\mathbb{D}^{*})\backslash \mathcal{B}(\mathbb{D}^{*})\\};

\draw[->, thick] (ka) edge node[auto] {$\beta$} (kb)
                         (ka) edge node[auto] {$p$} (kc)
                         (kc) edge node[auto] {$\widehat{\beta}$} (kd)
                         (kb) edge node[auto] {$P$} (kd)
                         ;
\end{tikzpicture}\\
\end{center}
%%%%%%%%%%%%%

For an arbitrary open subset $V \subset T_b$, we have 
\begin{equation*}
p^{-1}(p(V)) = \bigcup_{\tau \in T_v} R_{\tau}(V).
\end{equation*}
This shows that $p$ is an open map. In the same way,  
for an arbitrary open subset $U \subset \mathcal{B}(\mathbb{D}^{*})$, we have
\begin{equation*}
P^{-1}(P(U)) = \bigcup_{\varphi \in \mathcal{B}_0(\mathbb{D}^{*})} (U + \varphi).
\end{equation*}
This shows that $P$ is an open map. 
Moreover, the Bers embedding $\beta: T_b \to \mathcal{B}(\mathbb{D}^{*})$ is a homeomorphism from $T_b$ onto its image. Thus, $\widehat{\beta}$ is open and continuous. Combined with the injectivity of $\widehat{\beta}$, this implies that $\widehat{\beta}$ is a homeomorphism of $T_v\backslash T_b$ onto its image. 
\end{proof}

Concerning biholomorphic automorphisms of $p(T_b)=T_v \backslash T_b$
with respect to its complex structure, we have the following.
This kind of arguments are well-known in the theory of asymptotic Teichm\"uller spaces.

\begin{corollary}\label{biholomorphic}
For every $\tau \in T_b$, the biholomorphic automorphism $R_\tau$ of $T_b$ induces a
biholomorphic automorphism $\widehat R_\tau$ of $p(T_b)$ that satisfies
$p \circ R_\tau=\widehat R_\tau \circ p$.
\end{corollary}
\begin{proof}
For each $\sigma \in T_b$, we have that $R_\tau(T_v \ast \sigma)=T_v \ast (\sigma \ast \tau^{-1})$.
This shows that the correspondence $[\sigma] \mapsto [\sigma \ast \tau^{-1}]$ is well-defined to be
a map $\widehat R_\tau: p(T_b) \to p(T_b)$ that satisfies $p \circ R_\tau=\widehat R_\tau \circ p$.
By considering the inverse mapping $R_\tau^{-1}=R_{\tau^{-1}}$,
we see that $\widehat R_\tau$ is bijective. In the same way as the proof of Corollary \ref{quotient},
$\widehat R_\tau$ is shown to be a homeomorphism. For the statement, it suffices to prove that  
$\widehat R_\tau$ is holomorphic. 

We may identify $T_b$ with the domain $\beta(T_b)$ in $\mathcal{B}(\mathbb{D}^{*})$.
The conjugate $\widetilde R_\varphi=\beta \circ R_\tau \circ \beta^{-1}$ for $\varphi=\beta(\tau)$ 
is a biholomorphic automorphism of $\beta(T_b) \subset {\mathcal B}(\mathbb D^*)$.
We use its projection $\widehat R_\varphi$ to $P(\beta(T_b))=\widehat \beta(p(T_b))$ as a replacement of $\widehat R_\tau$,
which satisfies $P \circ \widetilde R_\varphi=\widehat R_\varphi \circ P$.
Let $\phi_1, \phi_2 \in \beta(T_b)$ with $\phi_1-\phi_2 \in {\mathcal B}_0(\mathbb D^*)$ and
let $\psi_1, \psi_2 \in {\mathcal B}(\mathbb D^*)$ with $\psi_1-\psi_2 \in {\mathcal B}_0(\mathbb D^*)$.
The derivative of $\widetilde R_\varphi$ satisfies
\begin{equation*}
\begin{split}
d_{\phi_1}\widetilde R_\varphi(\psi_1)&=\lim_{t \to 0} \frac{1}{t}(\widetilde R_\varphi(\phi_1+t\psi_1)-\widetilde R_\varphi(\phi_1)),\\
d_{\phi_2}\widetilde R_\varphi(\psi_2)&=\lim_{t \to 0} \frac{1}{t}(\widetilde R_\varphi(\phi_2+t\psi_2)-\widetilde R_\varphi(\phi_2)),
\end{split}
\end{equation*}
where the limits refer to the convergence in the norm.
From this, we see that $d_{\phi_1}\widetilde R_\varphi(\psi_1)-d_{\phi_2}\widetilde R_\varphi(\psi_2)$ belongs to ${\mathcal B}_0(\mathbb D^*)$
because ${\mathcal B}_0(\mathbb D^*)$ is closed and
\begin{equation*}
\begin{split}
&\quad\ \{\widetilde R_\varphi(\phi_1+t\psi_1)-\widetilde R_\varphi(\phi_1)\}-
\{\widetilde R_\varphi(\phi_2+t\psi_2)-\widetilde R_\varphi(\phi_2)\}\\
&=\{\widetilde R_\varphi(\phi_1+t\psi_1)-\widetilde R_\varphi(\phi_2+t\psi_2)\}-\{\widetilde R_\varphi(\phi_1)-\widetilde R_\varphi(\phi_2)\}
\end{split}
\end{equation*}
belongs to ${\mathcal B}_0(\mathbb D^*)$.
Thus, for every $[\phi] \in P(\beta(T_b))$, a linear map 
$A_{[\phi]}^\varphi:{\mathcal B}_0(\mathbb D^*)\backslash {\mathcal B}(\mathbb D^*) \to {\mathcal B}_0(\mathbb D^*)\backslash {\mathcal B}(\mathbb D^*)$
is well-defined by $A_{[\phi]}^\varphi ([\psi])= [d_{\phi}\widetilde R_\varphi(\psi)]$.
This satisfies $A_{[\phi]}^\varphi \circ P=P \circ d_\phi \widetilde R_\varphi$.

The linear operator $A_{[\phi]}^\varphi$ is bounded and the operator norm satisfies 
$\Vert A_{[\phi]}^\varphi \Vert \leqslant \Vert d_\phi \widetilde R_\varphi \Vert$.
Indeed, for every $[\psi] \in {\mathcal B}_0(\mathbb D^*)\backslash {\mathcal B}(\mathbb D^*)$ and every $\varepsilon >0$,
we choose $\psi \in {\mathcal B}(\mathbb D^*)$ such that $P(\psi)=[\psi]$ and 
$\Vert \psi \Vert \leqslant \Vert [\psi] \Vert+\varepsilon$. Then,
\begin{equation*}
\Vert A_{[\phi]}^\varphi([\psi]) \Vert=\Vert P \circ d_\phi \widetilde R_\varphi(\psi) \Vert
\leqslant \Vert d_\phi \widetilde R_\varphi(\psi) \Vert \leqslant \Vert d_\phi \widetilde R_\varphi \Vert \cdot \Vert \psi \Vert
\leqslant \Vert d_\phi \widetilde R_\varphi \Vert(\Vert [\psi] \Vert+\varepsilon).
\end{equation*}
Moreover, since we may assume that $\Vert \psi \Vert \leqslant 2\Vert [\psi] \Vert$
in the above choice of $\psi$, we have that 
\begin{equation*}
\begin{split}
&\quad\ \Vert \widehat R_\varphi([\phi]+[\psi])-\widehat R_\varphi([\phi])-A_{[\phi]}^\varphi([\psi]) \Vert \\
&=\Vert P \circ \widetilde R_\varphi(\phi+\psi)-P \circ \widetilde R_\varphi(\phi)-P \circ d_\phi \widetilde R_\varphi(\psi) \Vert \\
&\leqslant \Vert \widetilde R_\varphi(\phi+\psi)-\widetilde R_\varphi(\phi)-d_\phi \widetilde R_\varphi(\psi) \Vert =o(\Vert [\psi] \Vert).
\end{split}
\end{equation*}
This implies that $\widehat R_\varphi$ is differentiable at every $[\phi] \in P(\beta(T_b))$ in every direction
$[\psi] \in {\mathcal B}_0(\mathbb D^*)\backslash {\mathcal B}(\mathbb D^*)$ with the derivative
$d_{[\phi]} \widehat R_\varphi([\psi])=A_{[\phi]}^\varphi([\psi])$.
\end{proof}

\section{The Carleson metric and its quotient} 

In this section, we consider translation-invariant metrics on the BMO Teich\-m\"ul\-ler space $T_b$ and its quotient space 
$p(T_b)=T_v \backslash T_b$.
A translation-invariant metric $m=m(x,v)$ defined on the tangent bundle of $T_b$ satisfies $R_\tau^* m=m$ for every $\tau \in T_b$.
Here, the pull-back $R_\tau^* m$ of the metric $m$ by the automorphism $R_\tau$ is given by $R_\tau^* m(x,v)=m(R_\tau(x),d_\tau R_\tau(v))$.
Since $T_b$ is a complex manifold, the Kobayashi and the Carath\'eodory metrics are examples of such metrics
although they are in fact invariant under all biholomorphic automorphisms of $T_b$.

We define the following translation-invariant metric on $T_b$ in a canonical way.
For simplicity, the metric is given in the Bers embedding $\beta(T_b)$. As before,
we use the conjugate of the right translation $R_\tau$ of $T_b$ for $\tau \in T_b$ by the Bers embedding $\beta$,
which is the biholomorphic automorphism $\widetilde R_\varphi=\beta \circ R_\tau \circ \beta^{-1}$ of $\beta(T_b)$
for $\varphi=\beta(\tau)$. The derivative $d_\phi \widetilde R_\varphi: \mathcal{B}(\mathbb{D}^{*}) \to \mathcal{B}(\mathbb{D}^{*})$ at any point $\phi \in \beta(T_b)$ is a bounded linear operator.

\begin{definition}
{\rm
A translation-invariant metric $m_C$ at any point $\varphi \in \beta(T_b) \subset \mathcal{B}(\mathbb{D}^{*})$ 
and for any tangent vector $\psi \in \mathcal{B}(\mathbb{D}^{*})$ is 
defined to be $m_C(\varphi,\psi)=\Vert d_\varphi \widetilde R_\varphi \psi \Vert_{\mathcal B}$. 
We call this metric $m_C$ the Carleson metric on the BMO Teich\-m\"ul\-ler space $T_b \cong \beta(T_b)$. The pseudo-distance induced by this metric is denoted by $d_C(\cdot, \cdot)$,
which we call the Carleson distance.
}
\end{definition}

We note that for a smooth curve $\gamma=\gamma(t)$ in $\beta(T_b) \subset \mathcal{B}(\mathbb{D}^{*})$ 
with parameter $t \in [a,b]$,
its length $l_C(\gamma)$ is defined by the upper integral as
$$
l_C(\gamma)=\overline{\int_a^b} m_C(\gamma(t),\dot \gamma(t))dt.
$$
Then, the Carleson distance $d_C(\varphi_1,\varphi_2)$ is the infimum of $l_C(\gamma)$ taken over all smooth curves $\gamma$ connecting
$\varphi_1$ and $\varphi_2$.

Here is a list of intended results on the Carleson metric.
Concerning the classical case of the Teich\-m\"ul\-ler metric, we refer to
the work by Earle, Gardiner, and Lakic \cite{EGL}.
\begin{enumerate}
\item
A similar metric to the Carleson metric $m_C$ is introduced by considering extremal Beltrami coefficients in $\mathcal{M}(\mathbb D)$,
which is metrically equivalent to $m_C$.
\item
A predual space to $\mathcal{L}(\mathbb D)$ or $\mathcal{B}(\mathbb{D}^{*})$ is characterized and utilized to consider the metric.
\item
Under a certain smoothness of $m_C$, the BMO Teich\-m\"ul\-ler space $T_b$ is
equipped with a Finsler structure.
\item
The quotient metric of $m_C$ provides a Finsler structure for $p(T_b)$.
\item
The Carleson distance can be compared with the distance induced by the BMO norm.
\item
There is also a certain inequality between $m_C$ and the Kobayashi metric on $T_b$.
\end{enumerate}

In this section, we only prove that the Carleson metric restricted to
$T_v$ is continuous. In the next section, 
we prove that $T_b$ is complete with respect to the Carleson distance and certain relations
between the Carleson distance and the Teich\-m\"ul\-ler-Kobayashi distance.
 
Let $U(r) \subset \mathcal{B}(\mathbb{D}^{*})$ and $U^{\infty}(r) \subset B(\mathbb{D}^{*})$ denote the open balls of radius $r$ centered at the origin. We set $\delta_0 = 2/L,$ where $L$ is the absolute constant satisfying the condition $\Vert \varphi \Vert_{B} \leqslant L \Vert \varphi \Vert_{\mathcal{B}}$ as in \cite[Lemma 4.1]{SW}. More precisely, \cite{SW} handles the case for $\mathbb D$,
but if we note that $z^4\varphi(z)$ is a holomorphic function in $\mathbb D^*$ for every $\varphi \in B(\mathbb D^*)$,
the maximum principle justifies its arguments and the constant $L$ can be computed explicitly.
Then $U(\delta_0) \subset U^{\infty}(2) \subset \beta(T)$. 

We recall that a holomorphic local section of $\Phi:M(\mathbb D) \to B(\mathbb D^*)$ at the origin $0 \in B(\mathbb{D}^{*})$ can be given explicitly by Ahlfors and Weill \cite{AW}. The following form is the adaptation to the unit disk case.  For every $\varphi \in U^{\infty}(2)$, let 
\begin{equation*}
\sigma(\varphi) (z) =  -\frac{1}{2}\rho_{\mathbb{D}^{*}}^{-2}(z^{*}) (zz^{*})^2\varphi(z^{*}).
\end{equation*}
Then, $\nu(z) = \sigma(\varphi)(z)$ belongs to $M(\mathbb{D})$ and satisfies $\Phi(\nu) = \varphi$. Here, 
$z^{*} = 1/ \bar{z} =j(z) \in \mathbb{D}^{*}$ is the reflection of 
$z \in \mathbb{D}$ with respect to $\mathbb{S}$. Hence, $\sigma: U^{\infty}(2) \to M(\mathbb{D})$ is a holomorphic local section of $\Phi$ around $0$. Noting that $|\sigma(\varphi)(z)| = \frac{1}{2}|\varphi(z^{\ast})|(|z^{\ast}|^2 - 1)^2$, we see that $\Vert \sigma(\varphi) \Vert_{\infty} = \frac{1}{2}\Vert \varphi \Vert_{B} \leqslant \frac{1}{2}L\Vert \varphi \Vert_{\mathcal{B}}$ for every $\varphi \in U(\delta_0)$. Moreover,
we have that
$$
\lambda_{\sigma(\varphi)} = \frac{1}{4}\big(|\varphi(j(z))|^2(|j(z)|^2 - 1)^3 \big)\,\left|j_{\bar{z}}(z)\right|dxdy, 
$$
which means that $\lambda_{\sigma(\varphi)}$ is the pull-back of the Carleson measure $\frac{1}{4}|\varphi(z^{\ast})|^2(|z^{\ast}|^2 - 1)^3 dx^{\ast}dy^{\ast}$ on $\mathbb{D}^{\ast}$ by $j(z)=z^{\ast}$. Hence, 
$$\Vert \sigma(\varphi) \Vert_{\ast} = \Vert \sigma(\varphi) \Vert_{\infty} + \Vert \lambda_{\sigma(\varphi)}\Vert_{c}^{\frac{1}{2}} \leqslant \frac{1}{2}(L + M)\Vert \varphi \Vert_{\mathcal{B}}.$$
Here, the constant $M$ is the norm of the pull-back operator as in  \cite[Theorem 3.4]{WZ}. 
We can also obtain such a constant $M$ by a direct computation of the Carleson norm of $\lambda_{\sigma(\varphi)}$.
Consequently, $\sigma$ is well-defined from $U(\delta_0)$ onto its image in $\mathcal{M}(\mathbb{D})$ and the operator norm of $\sigma$ is bounded by $\frac{1}{2}(L + M)$. By the linearity of $\sigma$, we see that $\Vert d\sigma \Vert \leqslant \frac{1}{2}(L + M)$. 

We borrow the following discussion from Takhtajan and Teo \cite{TT}. 

\begin{lemma}\label{convergence}
For $\nu = \sigma(\varphi)$, $f^{\nu}$ and $(f^{\nu})^{-1}$ are bi-Lipschitz continuous under the Poincar\'e metric if 
$\varphi \in U^{\infty}(2)$ belongs to a small neighborhood of the origin of $B(\mathbb{D}^{*})$, 
and $\partial f^{\nu}$ and $\partial (f^{\nu})^{-1}$ converge locally uniformly to $1$ as $\varphi \to 0$ in $B(\mathbb{D}^{*})$.
\end{lemma}
\begin{proof}
It is proved in \cite[Lemma 2.5]{TT}
that for every $\varepsilon > 0$, there exists $0 < \delta < 1$ such that for all 
$\nu \in \sigma(U^\infty(2))$
with $\|\nu\|_{\infty} < \delta$, we have that
\begin{equation*}\label{Teo}
 \bigg|\frac{|\partial f^{\nu}(z)|^2}{(1 - |f^{\nu}(z)|^2)^2} - \frac{1}{(1 - |z|^2)^2}  \bigg| < \frac{\varepsilon}{(1 - |z|^2)^2}
\end{equation*}
for every $z \in \mathbb{D}$, 
and the same inequality holds for $(f^{\nu})^{-1}$. 
Then, 
\begin{equation*}
\sqrt{1 - \varepsilon} < \frac{1 - |z|^2}{1 - |f^{\nu}(z)|^2} |\partial f^{\nu}(z)| < \sqrt{1 + \varepsilon}.   
\end{equation*}
This proves the first statement. Combined with the fact that $f^{\nu}$ converges to the identity map 
uniformly on $\mathbb{D}$ as $\varphi \to 0$, this inequality also proves the second statement.
\end{proof}

The continuity of the Carleson metric in the special case is obtained as follows.

\begin{theorem}\label{cont}
 The Carleson metric $m_C$ is continuous on the VMO Teich\-m\"ul\-ler space $T_v$.
\end{theorem}
\begin{proof}
By the invariance, the continuity of $m_C$ follows from that at the origin;
$m_C(\varphi,\psi) \to m_C(0,\psi_0)$ as $(\varphi,\psi) \to (0,\psi_0)$.
Moreover,
by 
\begin{equation*}
\|d_{\varphi}\widetilde R_{\varphi}(\psi) - \psi_0\|_{\mathcal{B}} \leqslant \|d_{\varphi}\widetilde R_{\varphi}(\psi) - \psi\|_{\mathcal{B}} +  \|\psi - \psi_0\|_{\mathcal{B}},  
\end{equation*}
it suffices to show that for each tangent vector $\psi \in\mathcal{B}_0(\mathbb{D}^{*})$, 
$\|d_{\varphi} \widetilde R_{\varphi}(\psi) - \psi\|_{\mathcal{B}}$
converges to $0$ as $\varphi$ tends to 0 in $\mathcal{B}(\mathbb{D}^{*})$. The derivative $d_{\varphi} \widetilde R_{\varphi}$ can be decomposed into $d_{\varphi} \widetilde R_{\varphi} = d_0\Phi \circ d_{\nu}r_{\nu}\circ d_{\varphi}\sigma$ for 
$\nu = \sigma (\varphi)$. 
Then, 
\begin{equation*}
 d_{\varphi} \widetilde R_{\varphi}(\psi)(z) = -\frac{6}{\pi} \int_{\mathbb{D}} \frac{d_{\nu}r_{\nu}(\mu)(w)}{(w - z)^4} dudv 
 \qquad (z \in \mathbb{D}^{*})
\end{equation*}
for $\mu = d\sigma(\psi)$. 
We conclude that 
\begin{equation*}
\begin{split}
 &\quad\ |d_{\varphi} \widetilde R_{\varphi}(\psi)(z) - \psi(z)|^2\rho_{\mathbb D^*}^{-3}(z) \\
 &\leqslant A \int_{\mathbb{D}}\frac{|d_{\nu}r_{\nu}(\mu)(w) - \mu(w)|^2}{1 - |w|^2}\frac{(1 - |w|^2)(|z|^2 - 1)}{|w - z|^4} dudv\\
 \end{split}
\end{equation*}
for some absolute constant $A>0$.

We consider the pull-back of this Carleson measure by the reflection $j$ with respect to $\mathbb S$.
It holds that 
\begin{equation*}\label{R}
\begin{split}
 & \quad|d_{\varphi}\widetilde R_{\varphi}(\psi)(j(\zeta)) - \psi(j(\zeta))|^2\rho_{\mathbb D^*}^{-3}(j(\zeta))|j_{\bar{\zeta}}(\zeta)| \\
 &\leqslant A \int_{\mathbb{D}}|d_{\nu}r_{\nu}(\mu)(w) - \mu(w)|^2\rho_{\mathbb D}(w)\frac{(1 - |w|^2)(1 - |\zeta|^2)}{|\bar{w}\zeta - 1|^4} dudv\\
 \end{split}
\end{equation*}
for $z=j(\zeta)$ $(\zeta \in \mathbb D)$. 
We will show that the Carleson norm of $|d_{\nu}r_{\nu}(\mu)(w) - \mu(w)|^2\rho_{\mathbb D}(w)$ converges to $0$ as $\varphi \to 0$ in $\mathcal{B}(\mathbb{D}^{*})$. Then, by \cite[Lemma 11]{CZ} and \cite{Zi89} (also see \cite[Theorem 1.1]{WZ}), 
the Carleson norm of $|d_{\varphi} \widetilde R_{\varphi}(\psi)(z) - \psi(z)|^2\rho_{\mathbb D^*}^{-3}(z)$ converges to $0$,
which implies that $\|d_{\varphi} \widetilde R_{\varphi}(\psi) - \psi\|_{\mathcal{B}} \to 0$ as $\varphi \to 0$ in $\mathcal{B}(\mathbb{D}^{*})$. 

By computation, we see that 
\begin{equation*}
 d_{\nu}r_{\nu}(\mu)(w) = \frac{\mu(\zeta)}{1 - |\nu(\zeta)|^2} \frac{\partial f^{\nu}(\zeta)}{\overline{\partial f^{\nu}(\zeta)}} 
\end{equation*}
for $w = f^{\nu}(\zeta)$. Hence,
\begin{equation*}
\begin{split}
|d_{\nu}r_{\nu}(\mu)(w)  - \mu(w)|^2 
&\leqslant 2|d_{\nu}r_{\nu}(\mu)(w)|^2 + 2|\mu(w)|^2\\
&\leqslant 2\frac{1}{(1 - \|\nu\|_{\infty}^{2})^2}|\mu \circ (f^{\nu})^{-1}(w)|^2 + 2|\mu(w)|^2.
\end{split}
\end{equation*}
We note that $\psi \in \mathcal{B}_0(\mathbb{D}^{*})$ implies $\mu \in \mathcal{M}_0(\mathbb{D})$. By \cite[Proposition 3.5]{TWS},
we also have $\mu \circ (f^{\nu})^{-1} \in \mathcal{M}_0(\mathbb{D})$. 
Then, for any $\varepsilon > 0$, we can choose some $0 < r < 1$ such that the Carleson norm of 
$|d_{\nu}r_{\nu}(\mu)(w)  - \mu(w)|^2\chi_{A_r}(w)\rho_{\mathbb D}(w)$ is less that $\varepsilon$. 
Here, $\chi_{A_r}$ denotes the characteristic function of $A_r=\{w \mid 1-r<|w|<1\}$.

Moreover, 
Lemma \ref{convergence} implies that $d_{\nu}r_{\nu}(\mu)(w)$ converges locally uniformly to $\mu(w)$ as $\varphi$ tends to $0$. Thus,  
$|d_{\nu}r_{\nu}(\mu)(w) - \mu(w)|^2 \rho_{\mathbb D}(w)< \varepsilon$
for every $w \in \overline{\Delta_r}={\mathbb D}-A_r$ if $\varphi$ is sufficiently close to $0$ in $\mathcal{B}(\mathbb{D}^{*})$. 
In this case, it is easy to see that the Carleson norm of 
$|d_{\nu}r_{\nu}(\mu)(w)  - \mu(w)|^2\chi_{\overline{\Delta_r}}\rho_{\mathbb D}(w)$ is less than $2\pi \varepsilon$. 
Therefore, the Carleson norm of 
$|d_{\nu}r_{\nu}(\mu)(w)  - \mu(w)|^2\rho_{\mathbb D}(w)$ is less than $(2\pi+1)\varepsilon$ if $\varphi$ is sufficiently close to $0$.
\end{proof}

By this theorem, we can say that the VMO Teich\-m\"ul\-ler space $T_v$ has a continuous Finsler structure with
the Carleson metric.

We close this section by mentioning the quotient metric on $p(T_b)=T_v \backslash T_b$ induced by $m_C$.
We note that $m_C$ is invariant under the group structure of $T_b$ (the transitive group action of $T_b$ is
isometric with respect to $m_C$)
and the projection $p$ is given by
taking the quotient of the subgroup $T_v \subset T_b$. Then, the quotient metric $\widehat m_C$ on 
$p(T_b) \cong \widehat \beta(p(T_b)) \subset \mathcal B_0(\mathbb D^*) \backslash \mathcal B(\mathbb D^*)$ is defined by
\begin{equation*}
\widehat m_C(\widehat \varphi,\widehat \psi)=\inf\{m_C(\varphi,\psi) \mid P(\varphi)=\widehat \varphi,\ P(\psi)=\widehat \psi\}
\end{equation*}
for any $\widehat \varphi \in \widehat \beta(p(T_b))$ and 
$\widehat \psi \in\mathcal B_0(\mathbb D^*) \backslash \mathcal B(\mathbb D^*)$.
Moreover, we see that $\widehat m_C$ is invariant under
every biholomorphic automorphism $\widehat R_\tau$ of $p(T_b)$ verified in Corollary \ref{biholomorphic}. 
The pseudo-distance induced by $\widehat m_C$ on $p(T_b)$ coincides with
\begin{equation*}
\widehat d_C(\widehat \varphi_1, \widehat \varphi_2)=\inf \{d_C(\varphi_1,\varphi_2) \mid P(\varphi_1)=\widehat \varphi_1,\ P(\varphi_2)=\widehat \varphi_2\},
\end{equation*}
and this is in fact a distance. See \cite{EGL2} and Remark \ref{distance} in the next section.

\section{Properties of the Carleson distance}

In this section, we prove further properties of the Carleson distance mentioned in the previous section.
First, we give the following estimate of the operator norm of the derivative 
$d_0 \Phi:{\mathcal L}(\mathbb D) \to {\mathcal B}(\mathbb D^*)$ explicitly.
This can be used alternatively in the proof of Theorem \ref{cont} to show the convergence of
the Carleson norm of $|d_{\varphi} \widetilde R_{\varphi}(\psi)(z) - \psi(z)|^2\rho_{\mathbb D^*}^{-3}(z)$.
We remark that this explicit estimate is not necessary for other arguments in this section,
but might serve as a refinement of the results.

\begin{proposition}
$\Vert d_0 \Phi \Vert \leqslant 24$.
\end{proposition}

\begin{proof}
The derivative $d_0 \Phi$ can be represented by 
$$
\varphi(z) = d_0\Phi(\mu)(z) = -\frac{6}{\pi}\int_{\mathbb{D}}\frac{\mu(w)}{(w - z)^4} dudv \qquad (z \in \mathbb{D}^{*}).
$$
Applying the Cauchy-Schwarz inequality and the equation
$$
\int_{\mathbb{D}}\frac{dudv}{|w - z|^4} =\pi \rho_{\mathbb{D}^{*}}^{2}(z),
$$
we have
$$
|\varphi(z)|^2\rho_{\mathbb{D}^{*}}^{-3}(z) \leqslant \frac{36(|z|^2 - 1)}{\pi} \int_{\mathbb{D}} \frac{|\mu(w)|^2}{|w - z|^4} dudv.
$$
This shows that for every $\zeta \in \mathbb{S}$, 
\begin{equation*}
\begin{split}
\int_{\Delta(\zeta, r)\cap \mathbb{D}^{*}} |\varphi (z)|^2 \rho_{\mathbb{D}^{*}}^{-3}(z) dxdy & \leqslant 
\frac{36}{\pi} \int_{\Delta(\zeta, r)\cap \mathbb{D}^{*}}\Big(\int_{\mathbb{D}\setminus \overline{\Delta(\zeta, 5r/3)}}\frac{|\mu(w)|^2}{|w - z|^4} dudv\Big)(|z|^2 - 1)dxdy \\
& + \frac{36}{\pi} \int_{\Delta(\zeta, r)\cap \mathbb{D}^{*}}\Big(\int_{\mathbb{D}\cap \Delta(\zeta, 5r/3)}\frac{|\mu(w)|^2}{|w - z|^4}dudv\Big)(|z|^2 - 1)dxdy, \\
\end{split}
\end{equation*}
where $\Delta(\zeta, r)$ denotes the disk with center $\zeta$ and radius $r \in (0,2)$. 

For the first term $I_1$ in the right-hand side of the inequality above,
we note that $|w - z| \geqslant 2r/3$ and that $\mathbb{D}\setminus \overline{\Delta(\zeta, 5r/3)}$ is in the half-space
divided by a line passing through a given $z \in \Delta(\zeta, r)\cap \mathbb{D}^{*}$.
Moreover, $\Delta(\zeta, r)\cap \mathbb{D}^{*}$ is included in 
$S \cap \mathbb{D}^{*}$, where $S$ is a sector with center $0$, radius $1 + r$, and central angle at most
$\pi r$. Hence, we have that 
\begin{equation*}
 \begin{split}
I_1 & \leqslant \frac{36}{\pi}\Vert \mu \Vert_{\infty}^{2} \int_{S \cap \mathbb{D}^{*}}\Big( (|z|^2 - 1) \cdot \frac{1}{2}\int_{\{w \mid |w - z| \geqslant 2r/3\}}\frac{dudv}{|w - z|^4}\Big)dxdy\\
& \leqslant \frac{81}{2r^2}\Vert \mu \Vert_{\infty}^{2} \int_{S \cap \mathbb{D}^{*}}(|z|^2 - 1)dxdy\\
& \leqslant \frac{81}{2r^2}\Vert \mu \Vert_{\infty}^{2} \cdot \pi r \int_1^{1+r}(t^2-1)tdt\\
& \leqslant \frac{81\pi}{2} r\Big(\frac{r^2}{4}+r+1\Big) \Vert \mu \Vert_{\infty}^{2}
\leqslant 162\pi r \Vert \mu \Vert_{\infty}^{2}.
\end{split}   
\end{equation*}
For the second term $I_2$, similarly we have that
\begin{equation*}
 \begin{split}
I_2 & \leqslant \frac{36}{\pi} \int_{\mathbb{D}\cap \Delta(\zeta, 5r/3)}
\Big(\int_{S'\cap\mathbb{D}^{*}}\frac{|z|^2 - 1}{|w - z|^4}dxdy \Big)|\mu(w)|^2 dudv \\
& \leqslant \frac{36}{\pi} \int_{\mathbb{D}\cap \Delta(\zeta, 5r/3)}
\Big(\frac{3}{4}\int_{\{z \mid |w - z| \geqslant 1-|w|\}}\frac{2+r}{|w - z|^3}dxdy \Big)|\mu(w)|^2 dudv \\
& \leqslant 576 \cdot \frac{3}{5} \int_{\mathbb{D}\cap \Delta(\zeta, 5r/3)}
\frac{|\mu(w)|^2}{1 - |w|^2} dudv, \\
\end{split}
\end{equation*}
where $S'$ is a sector with center $w$ and central angle at most
$3\pi/2$.

From these estimates, we obtain that
\begin{equation*}
\Vert \varphi \Vert_{\mathcal{B}}^2 \leqslant \frac{1}{r}(I_1+I_2)
\leqslant 162 \pi \Vert \mu \Vert_{\infty}^{2}+576\Vert \lambda_{\mu} \Vert_c \leqslant 576\Vert \mu \Vert_*^2.
\end{equation*}
Thus, $\Vert \varphi \Vert_{\mathcal{B}} \leqslant 24\Vert \mu \Vert_{*}$,  which implies that $\Vert d_0 \Phi \Vert \leqslant 24$.
\end{proof}

In the following result, we obtain a locally uniform estimate for the operator norm $\Vert d_\varphi \widetilde R_\varphi \Vert$ when 
$\varphi \in \beta(T_b)$ is around the origin.
\begin{proposition}\label{upper}
The operator norm $\Vert d_\varphi \widetilde R_\varphi \Vert$ of the derivative 
$d_\varphi \widetilde R_\varphi:\mathcal{B}(\mathbb{D}^{*}) \to \mathcal{B}(\mathbb{D}^{*})$ at $\varphi$ 
is uniformly bounded from above and bounded away from zero for every $\varphi$ in an open ball 
$U(\delta_1) \subset \mathcal{B}(\mathbb{D}^{*})$ centered at the origin with some radius $\delta_1 \leq \delta_0$.
\end{proposition}
\begin{proof}
For the upper estimate, we decompose 
$\widetilde R_\varphi$ into $\widetilde R_\varphi = \Phi \circ r_{\nu} \circ \sigma$, 
where $\nu = \sigma (\varphi) \in \mathcal M ({\mathbb{D}})$. 
Then $d_{\varphi} \widetilde R_{\varphi}= d_0\Phi \circ d_{\nu}r_{\nu}\circ d_{\varphi}\sigma$. 
Here, the Ahlfors-Weill section $\sigma$ is linear with $\Vert d\sigma \Vert \leqslant \frac{1}{2}(L + M)$ as before,
and $d_0\Phi$ is a bounded linear operator with $\Vert d_0\Phi \Vert \leqslant 24$. Hence, it suffices to
consider $d_{\nu}r_{\nu}$.

As before, the derivative $d_{\nu}r_{\nu}$ in direction $\mu \in {\mathcal L}(\mathbb D)$ is given by  
\begin{equation*}
d_{\nu}r_{\nu}(\mu)(w)= \frac{\mu(\zeta)}{1 - |\nu(\zeta)|^2} 
\frac{\partial f^{\nu}(\zeta)}{\overline{\partial f^{\nu}(\zeta)}} 
\end{equation*}
for $w = f^{\nu}(\zeta)$. Then,
\begin{equation*}
\begin{split}
\frac{|d_{\nu}r_{\nu}(\mu)(w)|^2}{1-|w|^2}& \leqslant
\frac{1}{1-\Vert \nu \Vert_\infty^2}\frac{1-|(f^\nu)^{-1}(w)|^2}{1-|w|^2}|\partial (f^\nu)^{-1}(w)|^{-1}\\
& \times \frac{|\mu((f^\nu)^{-1}(w))|^2}{1-|(f^\nu)^{-1}(w)|^2}|\partial (f^\nu)^{-1}(w)|.
\end{split}
\end{equation*}
Here, the first factor of the right-hand side of the above inequality is uniformly bounded whenever
$\Vert \nu \Vert_\infty$ is less than some positive constant by Lemma \ref{convergence}.
In particular, for every $\varphi \in U(\delta_1)$ ($\nu=\sigma(\varphi)$) with some $\delta_1 \leq \delta_0$,  
this is uniformly bounded.
The second factor of the right-hand side of the above inequality is defined to be the pull-back 
$((f^\nu)^{-1})^*\lambda_\mu$
of the Carleson measure $\lambda_\mu=\lambda_\mu(\zeta)d\xi d\eta$ on $\mathbb D$ by $(f^\nu)^{-1}$.

By Semmes \cite[Lemma 4.8]{Se} (see also \cite[Proposition 3.5]{TWS}), we see that
$\Vert ((f^\nu)^{-1})^*\lambda_\mu \Vert_c \leqslant C \Vert \lambda_\mu \Vert_c$ for some
constant $C$ depending only on the bi-Lipschitz constant of $(f^\nu)^{-1}$ and
the strongly quasisymmetric constants of the boundary extension of $(f^\nu)^{-1}$.
In fact, the former constant depends on $\Vert \nu^{-1} \Vert_\infty$ $(\leqslant \Vert \nu^{-1} \Vert_*)$ by
Lemma \ref{convergence} as we have seen above, and
the latter constants depend only on $\Vert \nu^{-1} \Vert_*$ by 
Fefferman, Kenig, and Pipher \cite{FKP} (see also \cite[Lemma 4.3]{TWS}).
Furthermore, $\Vert \nu^{-1} \Vert_*$ can be estimated in terms of $\Vert \nu \Vert_*$. 
Therefore, we see that the constant $C$ is uniformly bounded.
This implies that there exists some constant $\widetilde C$ such that
the Carleson norms satisfy 
$\Vert \lambda_{d_{\nu}r_{\nu}(\mu)} \Vert_c \leqslant \widetilde C \Vert \lambda_\mu \Vert_c$ for every
$\varphi \in U(\delta_1)$ ($\nu=\sigma(\varphi)$).
Combined with the fact that 
$\Vert d_{\nu}r_{\nu}(\mu) \Vert_\infty \leqslant (1-\Vert \nu \Vert_\infty^2)^{-1}\Vert \mu \Vert_\infty$,
this proves that the operator norm $\Vert d_{\nu}r_{\nu} \Vert$ is uniformly bounded.

For the lower estimate of the operator norm $\Vert d_\varphi \widetilde R_\varphi \Vert$, we consider 
the upper estimate of
$\Vert d_0\widetilde R_\varphi^{-1} \Vert=\Vert d_\varphi \widetilde R_\varphi \Vert^{-1}$ by using the decomposition
$\widetilde R_\varphi^{-1} = \Phi \circ r_{\nu^{-1}} \circ \sigma$. Correspondingly, the derivative is 
$d_{0} \widetilde R_{\varphi}^{-1}= d_{\nu}\Phi \circ d_{0}r_{\nu^{-1}}\circ d_{0}\sigma$.
We know that $\Vert d\sigma \Vert \leqslant \frac{1}{2}(L + M)$ as before. Moreover, we have
the derivative $d_{0}r_{\nu^{-1}}$ in direction $\mu \in {\mathcal L}(\mathbb D)$ as
\begin{equation*}
\begin{split}
d_{0} r_{\nu^{-1}}(\mu)(w)&=\mu(\zeta)(1-|\nu^{-1}(\zeta)|^2) 
\frac{\partial (f^{\nu})^{-1}(\zeta)}{\overline{\partial (f^{\nu})^{-1}(\zeta)}}\\ 
&=\mu(f^\nu(w))(1-|\nu(w)|^2) 
\frac{\overline{\partial f^{\nu}(w)}}{\partial f^{\nu}(w)} \qquad (w=(f^{\nu})^{-1}(\zeta)).
\end{split}
\end{equation*}
Then, by a similar argument as before, we can prove that 
the operator norm $\Vert d_{0} r_{\nu^{-1}} \Vert$ is uniformly bounded for every
$\varphi \in U(\delta_1)$ ($\nu=\sigma(\varphi)$) by replacing $\delta_1$ with a smaller constant if necessary.

The locally uniform boundedness of the operator norm $\Vert d_{\nu}\Phi \Vert$ is a consequence of the holomorphy of $\Phi$.
This is a general argument but for the completeness, we review it here.
The derivative of $\Phi$ of the second order is the derivative of
$d\Phi:{\mathcal M}(\mathbb D) \to L({\mathcal L}(\mathbb D), {\mathcal B}(\mathbb D^*))$ given by the correspondence
$\nu \mapsto d_\nu \Phi$, where $L({\mathcal L}(\mathbb D), {\mathcal B}(\mathbb D^*))$ is the Banach space of
bounded linear operators ${\mathcal L} \to {\mathcal B}(\mathbb D^*)$ with respect to the operator norm.
Then, the property that $d\Phi$ is differentiable at $0$ is equivalent to the existence of a bounded linear operator
$A:\mathcal L \to L({\mathcal L}(\mathbb D), {\mathcal B}(\mathbb D^*))$ such that
$$
\Vert d_\nu \Phi-d_0 \Phi-A\nu \Vert=o(\Vert \nu \Vert_*)  \quad (\Vert \nu \Vert_* \to 0).
$$
Therefore, $\Vert d_\nu \Phi \Vert \leqslant \Vert d_0 \Phi \Vert +\Vert A \Vert \Vert \nu \Vert_* +o(\Vert \nu \Vert_*)$,
which yields the locally uniform boundedness of $\Vert d_\nu \Phi \Vert$.
\end{proof}

\begin{remark}\label{distance}
{\rm
To make the arguments in this section precise, we should note here that
the estimate of $\Vert d_\varphi \widetilde R_\varphi \Vert$ as in Proposition \ref{upper}
guarantees a locally uniform comparison of the metric with the norm of the Banach space.
Then, the pseudo-distance induced by
the Carleson metric is a distance and it defines the same topology as the original one on $T_b$.
}
\end{remark}

The completeness of the Carleson distance then follows from this  proposition.
\begin{theorem}\label{complete}
 The Carleson distance $d_C$ is complete on the BMO Teich\-m\"ul\-ler space $T_b$.
\end{theorem}
\begin{proof}
For any $\varphi_0$, $\varphi_1$ in $U(\delta_1) \subset \mathcal{B}(\mathbb{D}^{*})$, we choose the segment $\gamma = \{t\varphi_1 + (1 - t)\varphi_0\}_{t \in [0, 1]}$ connecting $\varphi_0$ and $\varphi_1$. Then, the Carleson length $l_C(\gamma)$ of $\gamma$ is given by 
$$
l_C(\gamma) = \overline{\int_0^1} m_C(t\varphi_1 + (1 - t)\varphi_0,\varphi_1 - \varphi_0)\, dt.
$$
Proposition \ref{upper} asserts that there is a constant $K$ such that
$$
\Vert d_{t\varphi_1 + (1 - t)\varphi_0} \widetilde R_{t\varphi_1 + (1 - t)\varphi_0}(\varphi_1 - \varphi_0) \Vert_{\mathcal{B}} \leqslant K \Vert \varphi_1 - \varphi_0 \Vert_{\mathcal{B}}. 
$$
Thus, we see that $l_C(\gamma) \leqslant K \Vert \varphi_1 - \varphi_0 \Vert_{\mathcal{B}}$. Then,
\begin{equation}
d_C(\beta^{-1}(\varphi_1), \beta^{-1}(\varphi_0)) \leqslant K \Vert \varphi_1 - \varphi_0 \Vert_{\mathcal B}. \tag{$\ast\ast$}
\end{equation}

We consider any Cauchy sequence in $(T_b, d_C)$. It suffices to consider its tail whose diameter can be arbitrary small. As the group 
of the right translations $\{R_\tau\}$ acts isometrically and transitively on $T_b$, we may assume that the tail of the Cauchy sequence is contained in $\beta^{-1}(U(\delta_1)))$. 
From the lower estimate of the derivative as in Proposition \ref{upper}, we see that
the Bers embedding of the Cauchy sequence is a Cauchy sequence, which is
a convergent sequence with respect to the norm $\Vert \cdot \Vert_{\mathcal{B}}$. 
Hence, $(\ast\ast)$ implies that the Cauchy sequence also converges with respect to $d_C$.
\end{proof}

We compare the Teich\-m\"ul\-ler metric and the Carleson metric.
The Teich\-m\"ul\-ler metric $m_T$ is given by defining a norm of a tangent vector
$\psi \in B(\mathbb D^*)$ at the base point of the universal Teich\-m\"ul\-ler space $T \cong \beta(T) \subset B(\mathbb D^*)$.
The norm of $\psi$ is the operator norm of the bounded linear functional
\begin{equation*}
H(\psi):A^1(\mathbb D^*) \to \mathbb C, \quad \varphi \mapsto \int_{\mathbb D^*} \varphi(z) \overline{\psi(z)} \rho_{\mathbb D^*}^{-2}(z)dxdy,
\end{equation*}
where $A^1(\mathbb D^*)$ is the Banach space of integrable holomorphic quadratic differentials on $\mathbb D^*$.
The operator norm $\Vert H(\psi) \Vert$ is comparable with $\Vert \psi \Vert_B$, and clearly
$\Vert H(\psi) \Vert \leqslant \Vert \psi \Vert_B$.
At any point $\varphi \in \beta(T)$, the Teich\-m\"ul\-ler metric is given by
$m_T(\varphi,\psi)=\Vert H(d_\varphi \widetilde R_\varphi \psi) \Vert$.
The distance induced by this metric is the Teich\-m\"ul\-ler distance $d_T$.
We consider the restriction of $d_T$ to the BMO Teich\-m\"ul\-ler space $T_b$. 

\begin{remark}
{\rm
For a smooth curve $\gamma=\gamma(t)$ $(a \leqslant t \leqslant b)$ in $\beta(T_b)$, the Teich\-m\"ul\-ler length of $\gamma$ is defined by
\begin{equation*}
l_T(\gamma)=\int_a^b m_T(\gamma(t),\dot \gamma(t))dt.
\end{equation*}
Then, the infimum of $l_T(\gamma)$ taken over all smooth curves in $T_b \cong \beta(T_b)$ connecting two points defines 
an inner distance $d^i_T$ between them, which clearly satisfies $d_T \leqslant d^i_T$.
}
\end{remark}

\begin{proposition}\label{comparison}
There exists a constant $L>0$ such that $m_T \leqslant L m_C$ on $T_b$.
Hence, $d_T \leqslant L d_C$ on $T_b$.
\end{proposition}
\begin{proof}
It was proved in \cite[Lemma 4.1]{SW} that there is some constant $L$ such that
$\Vert \psi \Vert_B \leqslant L \Vert \psi \Vert_{\mathcal B}$ 
for every $\psi \in \mathcal{B}(\mathbb D^*)$. 
Combined with $\Vert H(\psi) \Vert \leqslant \Vert \psi \Vert_B$, 
it follows that $\Vert H(\psi) \Vert \leqslant L \Vert \psi \Vert_{\mathcal B}$, and then
the assertion follows.
\end{proof}

It was shown by Fan and Hu \cite{FH} that the Kobayashi distance $d_K$ defined on the complex manifold $T_v$
coincides with the restriction of the Teich\-m\"ul\-ler distance $d_T$. In fact, $d_K=d_T^i=d_T$ on $T_v$.
Then, by Proposition \ref{comparison},
we have $d_K \leqslant L d_C$ on the VMO Teich\-m\"ul\-ler space $T_v$.
However, $d_T$ and $d_C$ are not comparable, that is, there is no inequality of the opposite direction
either for $T_b$ or for $T_v$. This is because the Carleson distance $d_C$ is complete in $T_b$ by Theorem \ref{complete}
and so is in the closed subspace $T_v$, but $d_T$ is not complete either in $T_b$ or in $T_v$.
In fact, the closure of $T_v$ in the universal Teich\-m\"ul\-ler space $(T,d_T)$ is $T_0$, the little subspace
given by vanishing Beltrami coefficients (asymptotically conformal maps), which contains an element not belonging to $T_b$.

\end{document}